\documentclass[reqno]{amsart}

\usepackage{amsfonts,amsmath, amssymb, amsthm}
\usepackage[foot]{amsaddr}
\usepackage{graphicx}
\usepackage{subcaption}
\usepackage{color}
\usepackage[hyperindex,breaklinks]{hyperref}
\usepackage{xspace}
\usepackage{framed}
\usepackage[ruled,noline]{algorithm2e}
\usepackage[utf8]{inputenc}
\usepackage{multirow}
\usepackage[a4paper]{geometry}
\usepackage[mathscr]{euscript}
\usepackage{tablefootnote}
\usepackage{setspace}
\usepackage{enumitem}  
\usepackage{booktabs}
\usepackage{adjustbox}
\usepackage{siunitx}

\newtheorem{theorem}{Theorem}[section]

\newtheorem{definition}[theorem]{Definition}

\newtheorem{lemma}[theorem]{Lemma}
\newtheorem{proposition}[theorem]{Proposition}
\newtheorem{corollary}[theorem]{Corollary}
\newtheorem{remark}[theorem]{Remark}

\hypersetup{colorlinks = true, urlcolor = blue, linkcolor = blue,
  citecolor = black}
\newcommand{\R}{\mathbb{R}}

\newcommand{\E}{\mathbb{E}}

\newcommand{\CR}{{\tt CR}}
\newcommand{\acr}{{\tt ACR}}
\newcommand{\apx}{{\tt apx}}
\newcommand{\vdp}{V^{\tt dp}}
\newcommand{\vce}{V^{\tt ce}}

\title[]{\Large
  \hspace{15pt} \MakeLowercase{\uppercase{M}ultiunit  \uppercase{I.I.D.}   
    \uppercase{P}rophet \uppercase{I}nequalities  
    via \uppercase{E}xtreme 
    \uppercase{V}alue  
    \uppercase{A}symptotics}}

\author{{\large J\MakeLowercase{ieming}
    K\MakeLowercase{ong}} \hspace{30pt} {\large
    K\MakeLowercase{arthyek} M\MakeLowercase{urthy}}}

\address{Daniel J. Epstein Department of Industrial \& Systems Engineering, University of Southern California}
\email{jiemingk@usc.edu,karthyek@usc.edu}
                                   
\begin{document}
\maketitle
\ \vspace{-40pt}
\begin{abstract}
\small   
 We study the i.i.d. $k$-selection prophet inequality problem, where a decision-maker sequentially observes $n$ independent nonnegative rewards and may accept at most $k$ of them without knowledge of future realizations. The objective is to maximize the expected total reward relative to that of a prophet who observes all rewards in advance. This problem captures the performance limits achievable in online resource allocation and underlies posted-price mechanisms in online marketplaces. We characterize the optimal welfare achievable relative to the prophet in terms of $k$ and the extreme value index of the reward distribution, in the asymptotic regime where the number of offers $n$ grows large. This optimal performance ratio turns out to be at least $1-\frac{\log k}{8k}[1+\varepsilon]$ for any $\varepsilon > 0$ and sufficiently large $k$, improving upon the respective, tight $1 - \frac{1}{\sqrt{2\pi k}}$ guarantee of static-threshold algorithms \cite{correa2021}. 

We additionally analyze the certainty-equivalent (CE) heuristic, a widely used online allocation algorithm known to yield optimal regret growth in $n$ when evaluated under the fluid scaling $k \propto n.$ Even in the absence of the fluid scaling $k \propto n,$  the CE heuristics's performance improves with $k$ to eventually match the leading order terms of the optimal dynamic program's performance ratio. A finer analysis nevertheless reveals that regret can be divergent and large relative to the optimal dynamic program when $n/k \to \infty$. This highlights the sensitivity in viewing the CE heuristic’s performance under the commonly adopted, though subjective, fluid scaling assumption. 
\end{abstract}

\section{Introduction}
\label{sec:Introduction}
\noindent

Prophet inequalities are fundamental results that capture the extent to which online resource allocation algorithms, operating under limited foresight, can approximate the performance of an all-knowing prophet. They have attracted significant attention in recent years  due to their close connection to posted-price mechanisms; \cite{Hajiaghayi, Haghiaghayi2018,chawla2010,correa2019}. Beyond the classical single-unit setting, multi-unit prophet inequalities and their generalizations arise naturally in a wide range of economic and operational contexts, including online advertising, revenue management, e-commerce fulfillment, and transportation logistics.

In this paper, we study the multi-unit i.i.d.\ prophet inequality problem, defined as follows. A decision-maker observes a sequence of \(n\) independent, nonnegative rewards arriving sequentially from a known distribution \(F\), and may accept at most \(k\) rewards, where \(k,n\) are positive integers satisfying \(k \leq n\). Upon the arrival of each reward, the decision-maker must make an irrevocable accept-or-reject decision before seeing future rewards. 

The performance of adaptive and non-adaptive algorithms in this setting has been studied from several perspectives, including the multi-secretary problem (see \cite{cayley1875,moser1956,arlotto2019}) and the design of posted-price mechanisms in online marketplaces (see \cite{alaei,correa2019,correa2021}). In a posted-price selling mechanism, a decision-maker seeks to sell \(k\) units of a product in an online marketplace. The \(n\) i.i.d.\ reward realizations in the multi-unit prophet inequality problem can be interpreted as the independent product valuations of a sequentially arriving stream of \(n\) potential buyers. The threshold set by an algorithm at any time for accepting an arriving reward then corresponds to the take-it-or-leave-it price posted by the seller.

A fundamental question in this setting is to understand how the expected total reward achievable by such a sequential (online) policy compares to that of a prophet who observes all rewards in advance and selects the top \(k\) among the $n$ realizations. We further aim to examine how conceptually and computationally simple acceptance policies perform relative to the optimal online policy.

\subsection{Known worst-case approximation guarantees in the i.i.d. setting}
The single-unit case \(k=1\), introduced in the i.i.d.\ setting by \cite{Hill-Kertz}, corresponds to a classical optimal stopping problem and is now well understood. By analyzing the optimal dynamic program associated with worst-case reward distributions, \cite{Hill-Kertz} showed that no online stopping rule can guarantee more than approximately \(0.745\) times the expected reward of the prophet. However, the tightness of this bound, in the sense of the existence of an online algorithm achieving this approximation ratio, was not established until the influential work of \cite{correa_2021_tightness_of_single_unit}.

For the multi-unit i.i.d.\ setting with \(k \geq 1\),  online algorithms are known to achieve at least  $1-k^k e^{-k}/k! \approx 1 - 1/\sqrt{2\pi k}$ fraction of the prophet's expected reward, independent of \(n\) and the underlying reward distribution; \cite{Arnosti2023,Beyhaghi2021,Chakraborty2010,dutting2019,Yan2011}. The fundamental nature of this problem has motivated extensive efforts to obtain tighter characterizations of this approximation ratio, captured via 
\[\CR_{k,n} := \inf_{F} \CR_{k,n}(F),\]
where \(\CR_{k,n}(F)\) denotes the ratio between the expected cumulative reward achieved by the optimal online policy and that of a prophet, for given \(k,n\) and reward distribution \(F\). The quantity \(\CR_{k,n}\), commonly referred to as the competitive ratio, captures the  performance guarantee achievable by online algorithms across all  probability distributions supported on the nonnegative real line. %Knowing $\CR_{k,n}$ translates into prophet inequality statements of the form, ``no online procedure can guarantee more than  $\CR_{k,n}$ fraction of the prophet's expected reward for all reward distributions; morever, there exists a probability distribution for which this guarantee is tight and cannot be improved, even for the optimal online algorithm". 
Recent investigations by \cite{jiang2023}, \cite{brustle2024}, and \cite{molina2025} have significantly advanced our understanding of \(\CR_{k,n}\), primarily through numerical characterizations, as well as through implicit asymptotic descriptions of \(\lim_{n \to \infty} \CR_{k,n}\). 
%in terms of coefficients for which a system of \(k\) nonlinear differential equations admits a solution. 
Specifically, \cite{brustle2024}, and \cite{molina2025} derive a characterization that reduces to 
\begin{align}
    \liminf_{n \to \infty} \CR_{k,n} \geq \sum_{i=1}^k \theta_i^\ast
    \label{eq:CR-Brustle}
\end{align}
where $\theta_1^\ast, \ldots, \theta_k^\ast$ are implicitly defined as parameter values for which a certain carefully derived system of $k$ nonlinear differential equations admits a solution.  This system generalizes the Hill–Kertz equation introduced in \cite{Hill-Kertz} for the single-unit case $k=1.$ Using these characterizations, \cite{brustle2024,molina2025} exhibit numerical lower bounds  that improve upon the factor $1-k^k e^{-k}/k!$ for $k \leq 5$  and large $n.$ 

The characterization \eqref{eq:CR-Brustle} does not readily translate into explicitly understandable lower bounds however, due to the limited understanding of the magnitude of the constants  $\theta_1^\ast, \ldots, \theta_k^\ast.$  Indeed, \cite{brustle2024} remarks on   the substantial difficulty involved in deriving explicit bounds for  $\theta_1^\ast, \ldots, \theta_{k-1}^\ast.$ Obtaining an analytical expression for $\CR_{k,n}$ remains an open problem, and likewise, gaining a more explicit understanding of how the asymptotic competitive ratio $\lim_{n \rightarrow \infty}\CR_{k,n}$ depends on $k$ is yet to be fully understood as well.  

%A more detailed discussion of these results is provided in Section~\ref{sec:lit-review}. 

%Despite this progress, obtaining an analytically explicit performance guarantee that improves upon the lower bound \(\CR_{k,n} \geq 1 - k^k e^{-k}/k!\) remains an open problem. Likewise, developing an explicit characterization of how the asymptotic competitive ratio \(\lim_{n \to \infty} \CR_{k,n}\) depends on $k$ is still not fully understood.

\subsection{Results on Instance-Dependent Asymptotic Competitive Ratio}
\subsubsection{A sharp characterization of the asymptotic competitive ratio}
Our first main result provides a precise characterization of the instance-dependent asymptotic competitive ratio
\[
\acr_k(F) := \lim_{n \to \infty} \CR_{k,n}(F),
\]
for the $k$-unit online selection problem with rewards drawn i.i.d.\ from a distribution $F$. We derive this characterization for all distributions satisfying the \emph{extreme value condition}, first introduced in this context by \cite{KennedyKertz1991} and subsequently studied in greater depth by \cite{correa2021,Livanos2025,correa2025}.

Just as how the central limit theorem characterizes the asymptotic behavior of sums of i.i.d.\ random variables, the extreme value condition governs the asymptotic behavior of the maximum of i.i.d.\ realizations. This condition holds for a broad class of distributions and forms the foundation of Extreme Value Theory (EVT); see, for example, \cite{deHaan2006extreme,resnick2008extreme}. Its relevance to prophet inequalities stems from the fact that both the prophet’s payoff and the performance of online algorithms are determined by the maximum of $n$ random variables and neighboring top order statistics.

Under the extreme value condition, formally defined in Section~\ref{sec:preliminaries}, we show that the asymptotic competitive ratio $\acr_k(F)$ depends only on $k$ and a single scalar parameter $\gamma \in \mathbb{R}$, known as the \emph{extreme value index} of the distribution $F$. While such a distribution-specific characterization does not directly yield prophet inequalities that hold uniformly across all distributions, it allows us to first identify the best achievable performance guarantees for all distributions sharing the same extreme value index.

\subsubsection{Worst-case asymptotic competitive ratio for large $k$}
We utilize the above conclusions to further establish that for any $\varepsilon > 0$, online algorithms can asymptotically achieve at least a
\[
1 - \frac{\log k}{8k}\,[1+\varepsilon]
\]
fraction of the prophet’s expected reward for all sufficiently large $k$, provided that the distribution $F$ satisfies the extreme value condition. Despite its asymptotic nature, this bound represents the strongest explicit analytical expression currently available for worst-case performance guarantees over a broad class of distributions. Moreover, this characterization is tight: we exhibit a distribution for which no online algorithm can surpass this asymptotic performance threshold.

\subsubsection{A discussion on the large $n$ assumption and instance-dependent characterizations}
The essential difference in our approach, when compared to the classical pursuit of understanding the worst-case competitive ratio $\CR_{k,n},$ can be described as follows: While the classical  competitive ratio analyses allow the worst-case distribution to depend on $n$ in potentially intricate ways, our approach fixes the distribution $F$ and studies performance  as the user base $n$ grows large. Understanding performance guarantees from this perspective is well motivated for applications such as large-scale online marketplaces featuring many potential buyers; see, for example, \cite{correa2021,correa2025}, the former of which includes an empirical case study demonstrating the practical relevance of this regime. From an analytical standpoint, the large-market assumption yields substantial tractability due to the statistical regularity of the top order statistics from the collection of $n$ i.i.d. realizations. Indeed, this regularity has been leveraged to derive asymptotic competitive ratio in the single-unit problem first in the seminal work of \cite{KennedyKertz1991} and later in \cite{Livanos2025}, and as well to examine the effectiveness of fixed-threshold (fixed-price) strategies in the multiunit problem by  \cite{correa2021} and \cite{correa2025}.  \cite{abdallah2024,abdallah2025} consider dynamic pricing for selling multiple units in a continuous time control setting and characterize the asymptotic behavior of the optimal pricing policy and its revenue in continuous time. This stream of literature, and our results, contribute to the broader line of study on instance-dependent approximation guarantees \cite{Arsenis2021,Cai2015,Feng2024}.

%Characterizing the asymptotic competitive ratio precisely addressing the $k$-unit prophet inequality problem from this large-market perspective remains, and the first part of this paper addresses this gap. 

\subsection{Results on comparison of optimal online performance with the CE heuristic} Can we quantify the performance loss incurred when a decision-maker restricts attention to computationally simpler and more interpretable policies, rather than employing the optimal online policy characterized by the dynamic programming recursion? The second part of this paper treats this question by considering the performance of two simpler, well-known policies: 1) the Certainty Equivalent (CE) heuristic, and 2) fixed-threshold policies, in relation to the optimal online policy. 

\subsubsection{Fixed threshold policies and the CE heuristic}
As the name suggests, a fixed-threshold policy selects a threshold \(T\) in advance and accepts the first \(k\) arrivals whose realized rewards exceed \(T\). Correa et al.~\cite{correa2021,correa2025} explicitly characterize the instance-dependent competitive ratio achievable by fixed-threshold policies. They further show that the best possible performance guarantee for fixed-threshold policies, taken over all distributions satisfying the extreme value condition, approaches \(1 - 1/{\sqrt{2\pi k}}\).

In contrast to the fixed-threshold policies, the Certainty Equivalent (CE) heuristic  dynamically adjusts its acceptance threshold so that the acceptance probability at each time matches the rate required to achieve the target of \(k\) acceptances over \(n\) arrivals (see Section~\ref{sec:ce} for a precise description). This intuitively appealing structure extends naturally to a much broader class of dynamic resource-constrained reward collection problems, including settings with multiple constraints.  We refer the reader to \cite{kleywegt,bumpensanti_wang,jasin_sinha,balseiro_2023} and the references therein for a comprehensive survey of the CE heuristic and its applications in network revenue management, dynamic pricing, e-commerce fulfillment, choice-based assortment optimization, and related problems. The budget-ratio policy, devised earlier in \cite{arlotto2019} for the multi-secretary problem, can be understood as a specialization of the broadly applicable CE heuristic to this specific setting.

The performance of the CE heuristic is typically  studied under the assumption that $k$ grows linearly proportional with $n,$ the so-called ``fluid-scaling" assumption. Examining the sequence of problems indexed by $n$ in this fluid scaling regime $k \propto n, n \rightarrow \infty,$ the CE heuristic has been shown to achieve regret growth rate (relative to the prophet benchmark) which matches that of known lower bounds in a variety of settings; see \cite{balseiro_2023,bray_LB}. 

\begin{center}
\emph{Does the strong performance of the CE heuristic, derived reliant on the fluid scaling assumption \(k \propto n\) as $n \rightarrow \infty,$ continue to hold when this assumption is relaxed and \(k\) and \(n\) are no longer coupled?}
\end{center}

Beyond its theoretical interest, this question is equally important from a practical perspective. For a decision-maker facing a fixed instance, say, for example, selecting \(k = 25\) rewards from a sequence of \(n = 600\) arrivals, it is far from clear that the fluid scaling assumption \(k \propto n\) is appropriate. While the fluid regime plays a central role in understanding the behavior of online algorithms, it is equally critical to examine their performance when the parameters \(k\) and \(n\) are decoupled and the fluid linear scaling assumption no longer applies.

\subsubsection{Results on the performance of CE heuristic}
The EVT based framework we derive for understanding the asymptotic competitive ratio scales well to the above challenge, and allows us to sharply characterize the performance of the CE heuristic for any fixed $k$ and $n$ growing large. To the best of our knowledge, this is the first result to allow continuous reward distributions in this setting and  characterize the performance of CE heuristic without the fluid scaling assumption $k \propto n.$ Our characterization complements the earlier seminal results due to \cite{arlotto2019} which  demonstrates uniformly bounded regret over all $\{(k,n): k \leq n\}$  in the presence of discrete random variables with finite support,

For $k = 1,$ the CE heuristic guarantees approximately only half the average reward collected by the optimal dynamic program. The performance of  CE heuristic improves with $k$ (see Table~\ref{tab:worst_case_ratios}), with the performance ratio growing to eventually match the leading order terms of the optimal dynamic program's performance for all large $k,$ even in the absence of the fluid scaling $k \propto n.$

A finer analysis nevertheless reveals divergent regret relative to the optimal dynamic program when $k=o(n)$ as $n \rightarrow \infty.$ In particular, we exhibit reward distributions for which the CE heuristic's regret can grow arbitrarily faster than the minimal regret growth rate. 
This divergent regret contrasts with the uniformly bounded regret in the case of finitely supported reward distribution \cite{arlotto2019}. In addition, it  underscores the sensitivity in viewing the CE heuristic’s performance under the commonly adopted, though subjective, fluid scaling assumption. 

Concurrent independent work by \cite{abdallah2026optimal} also investigates this question for the CE heuristic and shows that a modification based on the DP recursion and the extreme value index of the reward distribution to  eliminate the leading-order term in the divergent regret.

\subsubsection*{Paper organization} We provide a brief introduction of the extreme value condition in Section 2 and provide the performance guarantees of the optimal online policy and the CE heuristic in Sections 3 and 4. We develop insights from the numerical evaluation of the results in Section 5 and provide proofs of key results in Section 6. We conclude after discussing the implications of our results to the notion of competition complexity, see eg., \cite{correa2025} in Section 7.

\section{Preliminaries from Extreme Value Theory}
\label{sec:preliminaries} 
Let $X$ be a random variable with the distribution $F(x) = P(X \leq x).$ For any $n \geq 1,$ let $M_n = \max\{X_1,\ldots,X_n\}$ denote the maximum of $n$ i.i.d. copies of $X.$ Recall that the Central Limit Theorem characterizes the limiting distributions that  arise for normalized sums of i.i.d. random
variables. Along similar lines, the Extreme Value Theorem identifies the non-trivial
limiting distributions that may result for maxima of random variables  normalized as in
\[ \lim_{n \rightarrow \infty} \frac{M_n - b_n}{a_n},\]
for some suitable scaling sequences $\{a_n,b_n: n \geq 1\}.$ 
\begin{definition}[Extreme Value Condition]
\label{defn:EV-condition}
    We say that a distribution $F$ satisfies the extreme value condition if there exists sequences $\{a_n\}_{n \geq 1},\{b_n\}_{n \geq 1}$ such that the distribution of $(M_n - b_n)/a_n$ converges in distribution. 
\end{definition}

Just as how the limiting distribution in the central limit theorem must be a normal distribution or a stable distribution, any distribution that can arise in the right hand side of the limiting relationship, 
\begin{align}
  \label{MAX_LIM_DIST}
  \lim_{n \rightarrow \infty} P \left\{ \frac{M_n - b_n}{a_n} \leq x
  \right\} = \lim_{n \rightarrow \infty} F^n\left( a_n x + b_n \right)
  = G(x),  
\end{align}
must be of the form within the location-scale family of $G_\gamma(\cdot),$ where $\gamma \in \R,$ and 
\begin{align}
\label{EV-DISTS}
 G_\gamma(x) := \exp \left( -\left( 1 + \gamma x
  \right)^{-1/\gamma}\right), \quad\quad \text{ for all } x \text{ such that  } 1 + \gamma x > 0.
\end{align}
If $\gamma = 0,$ the right-hand side is interpreted as $\exp(-\exp(-x)).$ This conclusion is called the extreme value theorem; see \cite{Fisher_Tippett_1928},
\cite{Gnedenko1943}. Refer \cite{deHaan2006extreme} for an account of the diverse collection of  light-tailed ($\gamma \leq 0)$ and heavy-tailed distributions $(\gamma > 0)$ that satisfy the extreme value condition. The parameter $\gamma \in \R$ is called the \emph{extreme value index} of the distribution $F.$ Larger the parameter $\gamma,$ the heavier the tail CDF $\bar{F}(x) := P(X > x)$ is; and the case $\gamma \geq 1$ occurs only when $E[X] = \infty.$ 

It is important to note that $F$ itself need not be an extreme-value distribution of the form $G_\gamma$. This is analogous to how the central limit theorem does not require the underlying i.i.d.\ summands themselves to be normally distributed. For a given $\gamma \in \mathbb{R}$, distributions $F$ satisfying \eqref{MAX_LIM_DIST} are said to belong to the \textit{max-domain of attraction} of the extreme-value distribution $G_\gamma$. We denote the collection of all such distributions, namely the max-domain of attraction of $G_\gamma$, by $\mathcal{D}_\gamma$.

It is well known that the class of all distributions satisfying the extreme-value condition,
$\mathcal{D} := \bigcup_{\gamma \in \mathbb{R}} \mathcal{D}_\gamma,$
is dense in the space of univariate distribution functions; see \cite{Leonetti_Khorrami_Chokam_2022}. Consequently, restricting attention to distributions satisfying the extreme-value condition in Definition~\ref{defn:EV-condition} still permits a remarkably rich semiparametric class, since $\mathcal{D}$ is dense in the space of all univariate probability distributions.

%Hence our assumption that $F$ satisfies Definition \ref{defn:EV-condition} allows us to operate with a rich class of semi-parametric probability distributions that are dense in the space of all univariate probability distributions. 

%for details.

%to denote the collection of all distribution functions whose normalized maximum \(M_n\) converges to the extreme-value distribution \(G_\gamma\) in the sense of \eqref{EV-DISTS}. In other words, \(\mathcal D_\gamma\) is the max-domain of attraction of \(G_\gamma\). 

%\subsection{Distribution of the Maximum}
%\label{sec:evt}

%\subsection{Distribution of the $k$-Largest Order Statistics}

\section{Optimal Asymptotic Competitive Ratio}
\label{sec:dp}
Recall that in the $k$-unit prophet inequality problem, a decision-maker observes a sequence $\{X_1,\ldots,X_n\}$ of i.i.d. rewards  arriving sequentially from a known distribution \(F\). Constrained to accept at most $k$ arrivals from the 
sequence $\{X_1,\ldots,X_n\},$ the decision-maker must make an immediate accept-or-reject decision upon the arrival of each reward. 

Given $k \leq n,$ let $\vdp(n,k)$ denote the  expected reward of the optimal sequential (online) policy. If we let $\vdp(t,j)$ denote the  expected reward of the optimal policy when $j \leq k$ units are still to be accepted from the arrivals $\{X_{n-t+1},\ldots, X_n\},$ then one may write the dynamic programming (DP) equation 
\begin{align}
    \vdp(t,j) = \max_{\tau \geq 0} \left\{ \left( \mathbb{E}\left[ X \mid X \geq \tau \right] + \vdp({t-1,j-1}) \right) P(X \geq \tau) + \vdp(t-1,j)P(X < \tau)\right\} 
    \label{eq:dp}
\end{align}
for the optimal value function, coupled with the initial conditions $\vdp(0,j) = 0$ and $\vdp(t,0) = 0$ for all $j \leq k,t \leq n.$ Likewise, the expected reward collected by the Prophet can be described by
\begin{align}
    \mu_{n,k}\ :=\ \mathbb{E}\Big[\sum_{j=0}^{k-1} X_{n-j:n}
\Big] = \sum_{j=0}^{k-1} 
\mathbb{E}\Big[ X_{n-j:n} 
\Big], 
    \label{eq:prophet-reward}
\end{align}
where $X_{1:n}\le X_{2:n} \cdots\le X_{n:n}$ be the sorted list of  $\{X_1,\ldots,X_n\}$ denoting the order statistics. 

Equipped with the above notation, one can define the instance-dependent asymptotic competitive ratio for any given $k \geq 1$ and probability distribution $F$ by
\begin{align*}
    \acr_k(F) := \lim_{n \rightarrow \infty} \frac{\vdp(n,k)}{\mu_{n,k}}
\end{align*}

\subsection{A sharp characterization of the asymptotic competitive ratio}
Our first main result on the instance-dependent asymptotic competitive ratio, developed by analyzing the dynamic programming equation \eqref{eq:dp}, can be stated as follows. Throughout the paper, the notation $\Gamma(z) := \int_0^\infty t^{z-1}e^{-t}\mathrm{d}t$ is the gamma function. 

\begin{theorem}
    Let $F$ be a distribution over $\R^+$ that satisfies the extreme value condition. Then the optimal asymptotic competitive ratio attainable by the dynamic program solution 
    is  given by   
    \begin{align}
       \acr_k(F)  = 
       \begin{cases}
            (1-\gamma)^{1-\gamma} \frac{ v_k \Gamma(k)}{\Gamma(k+1-\gamma)}, &\text{if } \gamma  \in (0,1),\\
            1, \qquad\qquad  &\text{otherwise},\\
       \end{cases}
       \label{eq:acr-dp}
    \end{align}
    where $\gamma$ is the extreme value index of the distribution $F$ and the sequence $\{v_k: k \geq 1\}$ is obtained recursively from $v_1 = 1,$ and for any $k > 1,$ $v_k - v_{k-1}$ is the unique positive value of $x$ solving $x^{1/\gamma} + v_{k-1}x^{1/\gamma-1} - 1 = 0.$
\label{thm:acr-dp}
\end{theorem}

Observe that the $\acr_k(F)$ characterization in \eqref{eq:acr-dp} depends on the probability distribution $F$ only via its extreme value index. Moreover, it is smaller than 1 only if $\gamma \in (0,1),$ the case corresponding to $F$ being a heavy-tailed distribution of the regularly varying type with finite mean.  For the single-unit case, we obtain 
\begin{align*}
\acr_1(F) = \min\left\{\frac{(1-\gamma)^{-\gamma}}{\Gamma(1-\gamma)},1\right\} ,
\end{align*}
which is  at least 0.776 approximately. These observations specialized to the $k=1$ case match with the conclusions of \cite{KennedyKertz1991}. The $k > 1$ case does not bring out such explicit characterization due to the recursive characterization of $v_k,$ which in turn features a non-trivial dependence on $\gamma.$ The constant $v_k$ can be computed numerically with relative ease, by means of the recursion identified in the statement of Theorem \ref{thm:acr-dp}. The behavior of $\{v_k\}_{k \geq 1},$ for different values of $k,\gamma,$ is considered numerically in Section \ref{sec:num-exp}  and asymptotically in the subsequent Section \ref{sec:asymp}. The instance-dependent nature of the analysis manifests itself through an elementary recursion that is substantially easier to evaluate numerically, especially when compared to the task of identifying the parameters \(\theta_1^\ast,\ldots,\theta^\ast\) in \eqref{eq:CR-Brustle}. The latter corresponds to the inherently more challenging problem of establishing worst-case lower bounds that hold uniformly over all distributions.
 
\subsection{An understanding of $\acr_k$ for large $k$}
\label{sec:asymp}
Although the instance-dependent asymptotic competitive ratio does not admit a closed-form expression or explicit bound for all values of \(k\), it exhibits sufficient regularity to explicitly characterize the first-order dependence on \(\gamma\) and \(k\) for sufficiently large \(k\).

\begin{proposition}
Suppose $\gamma \in (0,1)$ in Theorem \ref{thm:acr-dp}. Then there exists a constant $M \in \R$ such that 
\begin{align}
    \left \vert \acr_k(F) - \left\{ 1 - \frac{\gamma(1-\gamma)}{2} \frac{\log k}{k} \right\} \right\vert \leq \frac{M}{k}, 
    \label{eq:dp-large-k}
\end{align}
for every $k \geq 1$.
\label{prop:dp-large-k}
\end{proposition}

In turn, Proposition \ref{prop:dp-large-k} allows us to derive an asymptotic lower bound that holds for all distributions satisfying the extreme value condition. This is noted in Corollary \ref{cor:dp-large-k} below. 

\begin{corollary}
Let $F$ be a distribution over $\R^+$ that satisfies the extreme value condition. Then given any $\varepsilon > 0,$ there exists $k_\varepsilon$ sufficiently large such that 
\begin{align*}
    \acr_k(F) \geq 1-\frac{\log k}{8k}[1+\varepsilon], \quad \text{for all }  k \geq k_\varepsilon, 
\end{align*}
irrespective of the extreme value index $\gamma.$
\label{cor:dp-large-k}
\end{corollary}
Thus, the performance guarantee obtainable with the optimal dynamic program turns out to be at least $1-\frac{\log k}{8k}[1+\varepsilon]$ for any $\varepsilon > 0$ and sufficiently large $k$, improving upon the respective, tight $1 - \frac{1}{\sqrt{2\pi k}}$ guarantee achievable with fixed threshold algorithms \cite{correa2021}.  Similar to the fixed threshold setting, the worst-case performance is achieved when the extreme value index $\gamma = 1/2.$ 

In addition to the asymptotic guarantees that apply for large \(k\), we examine the quantity \(\acr_{n,k}\) identified in Theorem~\ref{thm:acr-dp} via a numerical evaluation of the constant \(v_k\) for different values of \(k,\gamma\); see Section~\ref{sec:num-exp}.

\section{Comparisons with the CE Heuristic} 
\label{sec:ce}
The objective of this section is to quantify the performance losses incurred,  relative to the oracle and the optimal dynamic program,  when a decision-maker restricts attention to simpler class of adaptive algorithms, such as the Certainty Equivalent (CE) heuristic. The CE heuristic is widely used in the rich literature on online resource allocation (see \cite{kleywegt,bumpensanti_wang,jasin_sinha,balseiro_2023} and the references therein) and is known to achieve optimal regret growth under the fluid scaling assumption \(k \propto n\). The budget-ratio policy, devised earlier in \cite{arlotto2019} for the multi-secretary problem, can be understood as a specialization of the broadly applicable CE heuristic to this specific setting.

The idea behind Certainty Equivalent (CE) heuristic is to  dynamically adjust the acceptance threshold so that the acceptance probability at each time matches the rate required to achieve the target of \(k\) acceptances over \(n\) arrivals.  Specifically, suppose $F$ is continuous. Then when $j \leq k$ units are still to be accepted from the remaining $t$ arrivals $\{X_{n-t+1},\ldots, X_n\},$ %Then, upon observing the \((n-t+1)\)-th arrival, 
the CE heuristic accepts it if and only if its value exceeds the quantile $F^{\leftarrow}\!\left(1 - j/t \right)$
of the reward distribution \(F\). With the notation
\[F^{\leftarrow}(p) := \inf\{x : F(x) \geq p\}\]
for any $p \in [0,1],$ this rule ensures that the conditional probability of acceptance at time \(n-t+1\) is exactly \(j/t\), corresponding to the uniform acceptance rate required to exhaust the remaining capacity over the remaining $t$ arrivals. 
This intuitively appealing structure extends naturally to a much broader class of dynamic resource-constrained reward collection problems, including settings with multiple constraints.  We refer the reader to \cite{balseiro_2023} and the references therein for a comprehensive survey of the CE heuristic and its applications. 

Let $\vce(t,j)$ denote the expected reward that can be achieved with the  CE heuristic when $j \leq k$ acceptances need to be made over $t \leq n$ remaining arrivals. Then, due to the above described structure, the CE heuristic satisfies the recursion
\begin{align}
\vce(t,j)
&=
\Bigl(\,\mathbb{E}[X \mid X\ge \tau^{\tt{CE}}_{t,j}] + V^{\tt{CE}}(t-1,j-1)\Bigr)\,
P(X\ge \tau^{\tt{CE}}_{t,j})
\;+\;
\vce(t-1,j)\,
P(X< \tau^{\tt{CE}}_{t,j}),
\label{eq:ce}
\end{align}
for all $t \leq n, j \leq k;$ here 
\[
\tau^{\tt{CE}}_{t,j}
\;:=\;
F^{\leftarrow}\!\left(1-\frac{j}{t}\right).
\]
% \tau^{\mathrm{CE}}_{n,k}
% \;&:=\;
% F^{\leftarrow}\!\left(1-\frac{k}{n}\right).\\
% \label{eq:ce}
% \end{align}

Equipped with this notation, one can define the instance-dependent performance guarantee for the CE heuristic that holds for any given $k \geq 1$ and probability distribution $F$ as follows:
\begin{align*}
    \apx_k(F) := \lim_{n \rightarrow \infty} \frac{\vce(n,k)}{\mu_{n,k}},
\end{align*}
where $\mu_{n,k}$ is the Prophet's expected reward. 

\subsection{Asymptotic competitive ratio of the CE heuristic}
Our first result on the CE heuristic,  developed by analyzing the two-dimensional recursion \eqref{eq:ce}, can be stated as follows. 
\begin{theorem}
    Let $F$ be a distribution over $\R^+$ that satisfies the extreme value condition. Then the asymptotic competitive ratio attainable by the CE heuristic satisfies  
    \begin{align}
       \apx_k(F)  = 
       \begin{cases}
       \frac{\Gamma(k)\Gamma(k+1)}{\Gamma(k+1-\gamma)\Gamma(k+1+\gamma)} \sum_{r=1}^k \frac{\Gamma(r+\gamma)}{\Gamma(r+1)}r^{1-\gamma}  &\text{if } \gamma  \in (0,1),\\
           1, \qquad\qquad  &\text{if } \gamma \notin (0,1),  \end{cases}
           \label{eq:acr-ce}
    \end{align}
    where $\gamma$ is the extreme value index of the distribution $F.$ 
    \label{thm:acr-ce}
\end{theorem}
Similar to Theorem \ref{thm:acr-dp}, the performance ratio $\apx_k(F)$ in \eqref{eq:acr-ce} depends on the probability distribution $F$ only via its extreme value index. When $k = 1,$ the performance of CE heuristic guarantees only approximately 0.5 fraction of the Prophet's  average reward. It improves with $k,$ and offers approximately 0.9 fraction of the Prophet's worst-case reward, even for  $k$ as small as $10;$ please see Table \ref{tab:worst_case_ratios} in Section \ref{sec:num-exp} for a detailed study of how the worst-case approximation ratios improves with $k.$ For large values of $k,$ Proposition \ref{prop:ce-large-k}  below provides an understanding of the CE heuristic's performance guarantee. 

\begin{proposition}
Suppose $\gamma \in (0,1)$ in Theorem \ref{thm:acr-ce}. Then there exists a constant $M' \in \R$ such that 
\begin{align}
    \left \vert \apx_k(F) - \left\{ 1 - \frac{\gamma(1-\gamma)}{2} \frac{\log k}{k} \right\} \right\vert \leq \frac{M'}{k}, 
    \label{eq:dp-large-k}
\end{align}
for every $k \geq 1$.
\label{prop:ce-large-k}
\end{proposition}
Interestingly, the leading order terms in the expansion for the CE heuristic's performance guarantee matches with that reported for the optimal dynamic program in Section \ref{prop:dp-large-k}. The worst-case performance guarantee holds in Corollary \ref{cor:ce-large-k} then as a simple consequence.

\begin{corollary}
Let $F$ be a distribution over $\R^+$ that satisfies the extreme value condition. Then given any $\varepsilon > 0,$ there exists $k'_\varepsilon$ sufficiently large such that 
\begin{align*}
    \apx_k(F) \geq 1-\frac{\log k}{8k}[1+\varepsilon], \quad \text{for all }  k \geq k'_\varepsilon,
\end{align*}
irrespective of the extreme value index $\gamma.$ 
\label{cor:ce-large-k}
\end{corollary}
Being an adaptive algorithm, we see that the worst-case guarantee reported for CE heuristic in Corollary \ref{cor:ce-large-k} is considerably better than the tight $1-1/\sqrt{2\pi k}$ guarantee of fixed threshold algorithms, for  large values of $k.$

\subsection{Finer comparison between dynamic program and CE heuristic}
Though the leading order terms of $\acr_k(F)$ and $\apx_k(F),$ capturing the instance-specific performance guarantees match in Propositions \ref{prop:dp-large-k} and \ref{prop:ce-large-k},  
 finer analysis below reveals that the regret can be large and divergent when compared  to the optimal dynamic program.
\begin{theorem}
    Let $F$ be a distribution over $\R^+$  with finite mean and that satisfies the extreme value condition. Then the additional regret incurred by the CE heuristic, relative to the DP, is given by   
    \begin{align*}
    \lim_{k\to\infty}\ \lim_{n\to\infty}\ 
    \frac{\vdp(n,k)-\vce(n,k)}{k^{-\gamma}F^{\leftarrow}\left(1- \frac{1}{n} \right)}
    \;=\; c_\gamma .
    \end{align*}
    % \begin{align*}
    %  \vdp (n,k) - \vce (n,k)  =  c_\gamma k^{-\gamma}F^{\leftarrow}\left( 1- \frac{1}{n}\right) [1+ o(1)], \quad \text{ as } n \rightarrow \infty, 
    % \end{align*} 
     where $c_\gamma$ is a finite positive constant if $\gamma \in (0,1)$ and $c_\gamma = 0$ if $\gamma \le 0.$ Consequently, as $n \to \infty$, the CE heuristic's additional regret $\vdp (n,k) - \vce (n,k)$  is divergent  if the distribution $F$'s extreme value index  $\gamma \in (0,1).$ 
    \label{thm:regret-ce-dp}
\end{theorem}

\begin{proposition}
    Given any $\gamma \in (0,1),$ there exists a probability distribution $F$ satisfying the extreme value condition with index $\gamma$ and for which  $\vdp (n,k) - \vce (n,k) = \Omega((n/k)^{\gamma})$  whenever $k = o(n)$ as $n \rightarrow \infty.$
    \label{prop:regret-ce-dp-pareto}
\end{proposition}

In contrast to Theorem \ref{thm:regret-ce-dp}, Proposition \ref{prop:regret-ce-dp-pareto} identifies large regret even if $k$ is taken to grow with $n,$ as $n \rightarrow \infty.$ The rate of regret growth, in particular, can be significantly larger than the optimal $\Theta(\log n)$ rate identified under the fluid-scaling assumption, see eg., \cite{bray_LB, balseiro_2023} and references therein.   The divergent regret reported in Theorem \ref{thm:regret-ce-dp} and Proposition   \ref{prop:regret-ce-dp-pareto} contrasts starkly with the uniformly bounded regret in the case of finitely supported reward distribution \cite{arlotto2019}. In addition, it  underscores the sensitivity in viewing the CE heuristic’s performance under the commonly adopted, though subjective, fluid scaling assumption.

\section{Numerical Evaluation of the performance guarantees}
\label{sec:num-exp}

This section illustrates how the asymptotic competitive ratios depend on the budget level $k$ and the extreme-value
index $\gamma$. Throughout, we focus on $F$ belonging to the Fr\'echet domain of attraction, a nomenclature corresponding to the extreme value index $\gamma > 0.$ In particular, we study the approximation guarantees for $\gamma\in(0,1),$  the values for which the  performance loss in Theorems \ref{thm:acr-dp}-\ref{thm:acr-ce} are 
nontrivial. For each pair $(k,\gamma)$ we evaluate the closed-form expressions in
Theorem~\ref{thm:acr-dp} (DP) and Theorem~\ref{thm:acr-ce} (CE), and visualize the resulting ratios in
Figure~\ref{fig:heatmaps}. We also compute the worst-case ratios over $\gamma\in(0,1)$ and report them in
Table~\ref{tab:worst_case_ratios}.

\subsection{Numerical evaluation of $\acr_k(F)$ and $\apx_k(F)$ over $(k,\gamma)$}\label{subsec:heatmaps}

Figure~\ref{fig:heatmaps} displays heatmaps of the asymptotic competitive ratios $\acr_k(F)$ (DP vs.\ oracle) and
$\apx_k(F)$ (CE vs.\ oracle) as functions of $k$ and $\gamma$. Two qualitative patterns are immediate from Figure~\ref{fig:heatmaps}. 

First, for both DP and CE, performance improves rapidly with the budget level $k$.
In particular, the ratios approach $1$ quickly as $k$ increases, consistent with the large-$k$ expansion
\[
\acr_k(F),\ \apx_k(F)
=
1-\frac{\gamma(1-\gamma)}{2}\frac{\log k}{k}
+O\!\left(\frac{1}{k}\right)
\qquad (k\to\infty),
\]
proved in Propositions~\ref{prop:dp-large-k} and~\ref{prop:ce-large-k}. Visually, the heatmaps become nearly
uniformly close to $1$ once $k$ is moderate (e.g., $k\gtrsim 20$).

Second, the deterioration is concentrated in the heavy-tail boundary $\gamma\uparrow 1$ and for very small budgets.
This is particularly pronounced for the CE heuristic: for small $k$, its worst performance occurs extremely close
to $\gamma=1$, whereas the DP remains comparatively robust.
\begin{figure}[htbp]
  \centering
  \begin{subfigure}[b]{0.48\textwidth}
    \centering
    \includegraphics[height=5cm, keepaspectratio]{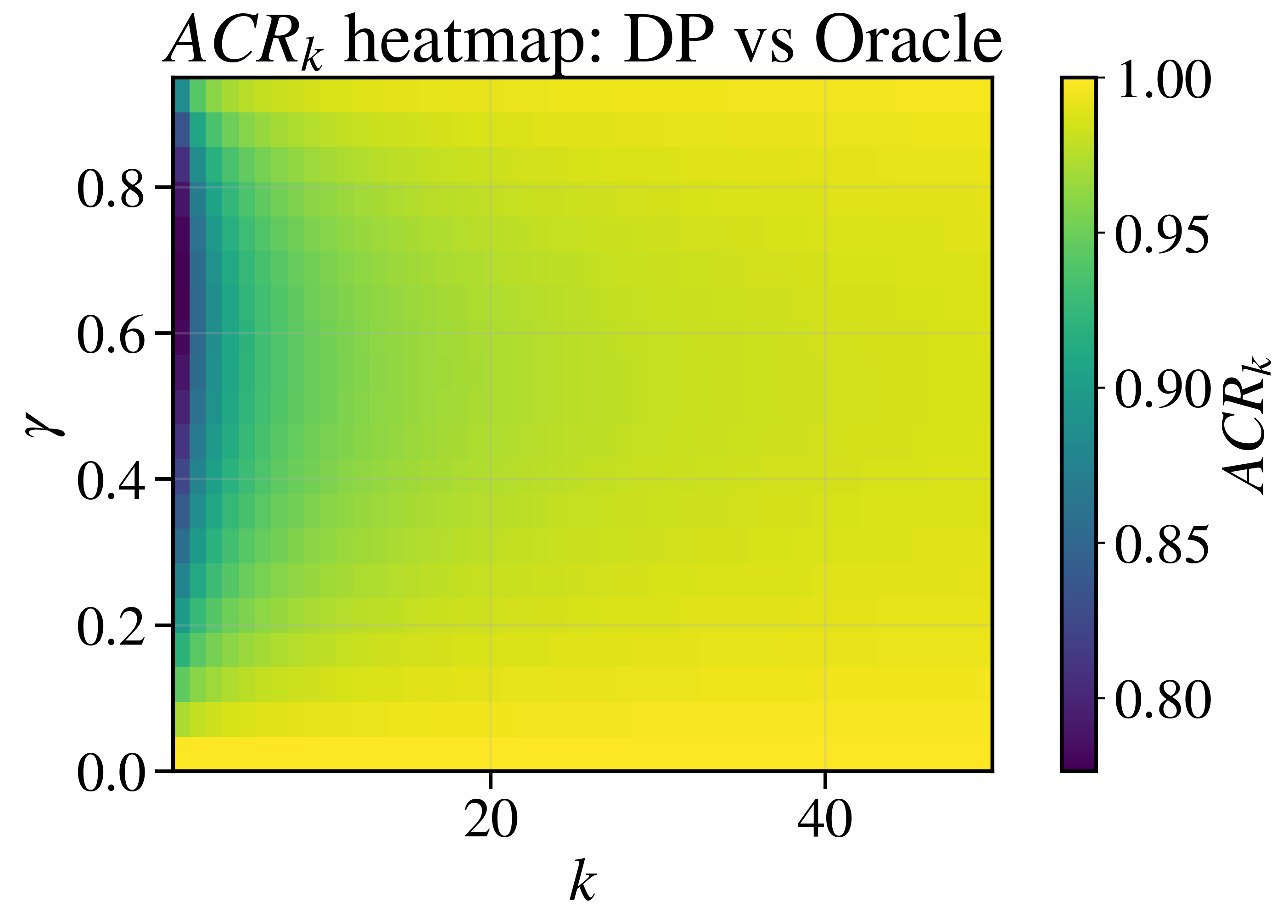}
    \caption{DP vs Oracle}
  \end{subfigure}
  \hfill
  \begin{subfigure}[b]{0.48\textwidth}
    \centering
    \includegraphics[height=5cm, keepaspectratio]{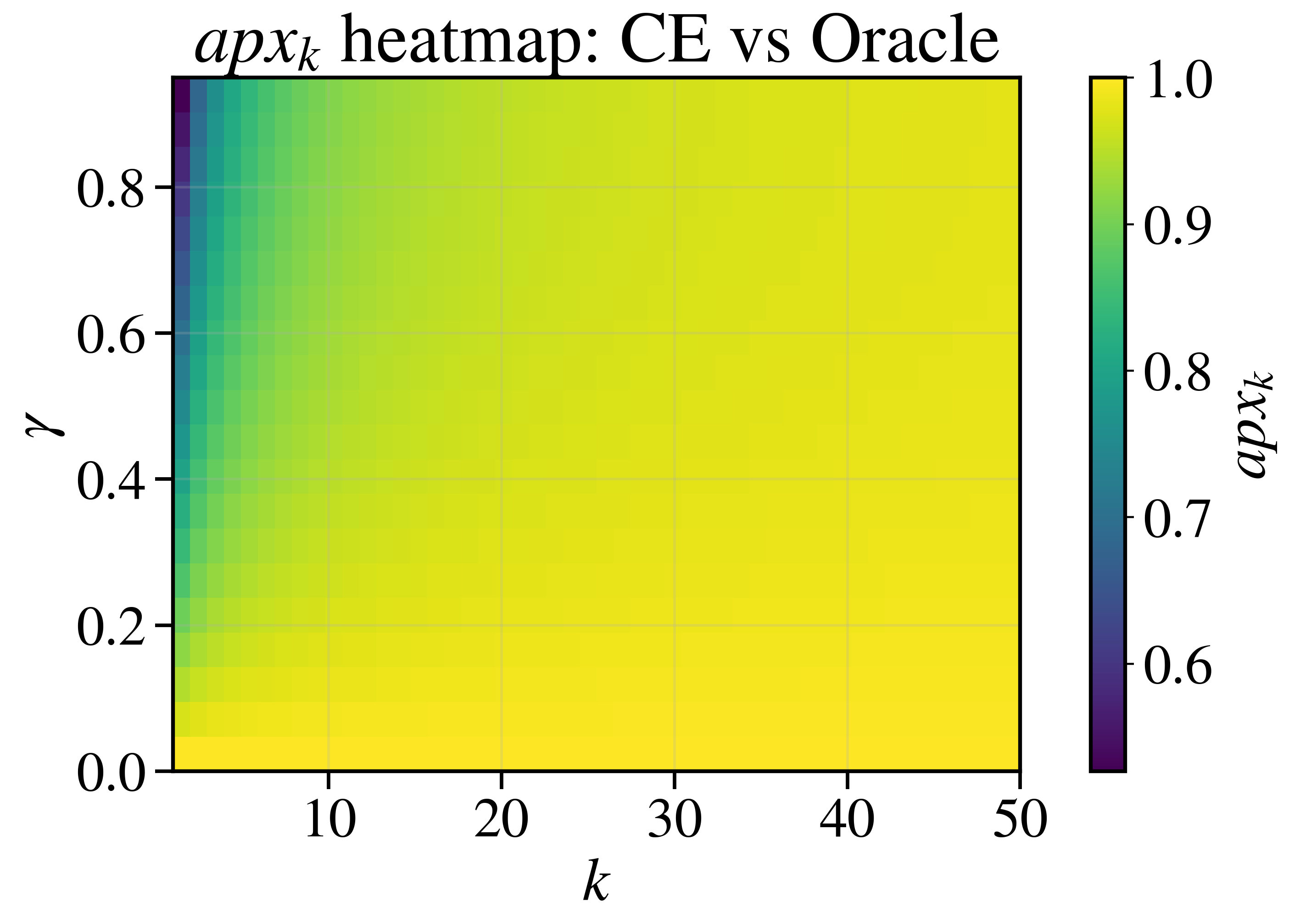}
    \caption{CE vs Oracle}
  \end{subfigure}
  \captionsetup{font=large}
  \caption{Heatmaps of the asymptotic ratios as a function of $k$ and $\gamma$.}
  \label{fig:heatmaps}
\end{figure}

\subsection{Worst-case ratios of $\acr_k(F)$ and $\apx_k(F)$ over $\gamma$}\label{subsec:worstcase}

To quantify the worst-case behavior, define
\[
\underline{\acr}^{\mathrm{DP}}_k:=\inf_{\gamma\in(0,1)}\acr_k(F),\qquad
\underline{\acr}^{\mathrm{CE}}_k:=\inf_{\gamma\in(0,1)}\apx_k(F),\qquad
\underline{\rho}_k:=\inf_{\gamma\in(0,1)}\frac{\apx_k(F)}{\acr_k(F)},
\]
and let $\gamma^\ast_{\mathrm{DP}},\gamma^\ast_{\mathrm{CE}},\gamma^\ast_{\mathrm{CE/DP}}$ denote points attaining these infima on the numerical grid.

Table~\ref{tab:worst_case_ratios} shows that $\underline{\acr}^{\mathrm{DP}}_k$ and $\underline{\acr}^{\mathrm{CE}}_k$
increase monotonically with $k$. For example, when $k=1$ the worst-case ratios are
\[
\underline{\acr}^{\mathrm{DP}}_1\simeq 0.776,\qquad
\underline{\acr}^{\mathrm{CE}}_1\simeq 0.501,
\]
while by $k=10$ they improve to
\[
\underline{\acr}^{\mathrm{DP}}_{10}\simeq 0.953,\qquad
\underline{\acr}^{\mathrm{CE}}_{10}\simeq 0.909,
\]
and by $k=200$ both are essentially indistinguishable from $1$
($\underline{\acr}^{\mathrm{DP}}_{200}\simeq 0.996$, $\underline{\acr}^{\mathrm{CE}}_{200}\simeq 0.995$).

The table also highlights a structural difference in where the worst case occurs.
For DP, the minimizing $\gamma^\ast_{\mathrm{DP}}$ lies in an intermediate range and drifts slowly downward with $k$
(e.g., from about $0.685$ at $k=1$ to about $0.525$ at $k=200$). This aligns with the fact that the leading
large-$k$ correction involves $\gamma(1-\gamma)$, whose maximum is attained at $\gamma=1/2$.
In contrast, for CE the worst case is attained near the heavy-tail boundary for small budgets
($\gamma^\ast_{\mathrm{CE}}\approx 0.999$ for $k\le 20$), and then shifts away from $1$ as $k$ grows
(e.g., $\gamma^\ast_{\mathrm{CE}}\approx 0.935$ at $k=50$ and $\gamma^\ast_{\mathrm{CE}}\approx 0.765$ at $k=200$).
This is consistent with CE being most sensitive to extremely heavy tails when the budget is very small.

Finally, the CE-to-DP ratio $\underline{\rho}_k$ confirms that the additional loss of CE relative to DP is mainly a
small-$k$ phenomenon. The worst-case $\mathrm{CE}/\mathrm{DP}$ ratio is about $0.504$ at $k=1$, but increases
rapidly with $k$ (e.g., $\underline{\rho}_{10}\simeq 0.909$ and $\underline{\rho}_{50}\simeq 0.980$),
indicating that CE becomes nearly as good as DP once the budget is moderate. Moreover,
$\gamma^\ast_{\mathrm{CE/DP}}$ is essentially at the boundary $\gamma\approx 1$ across all reported $k$,
reinforcing that the relative gap is driven by the heaviest-tail regime.
\begin{table}[htbp]
\centering

\caption{Worst-case performance ratios and the values of $\gamma^*$ where they occur}
\label{tab:worst_case_ratios}

\adjustbox{max width=\textwidth}{
\begin{tabular}{
c|
S[table-format=1.3]
S[table-format=1.3]|
S[table-format=1.3]
S[table-format=1.3]|
S[table-format=1.3]
S[table-format=1.3]
}
\toprule
$k$
& \multicolumn{1}{c}{Worst $\acr_k$ (DP/Prophet)}
& \multicolumn{1}{c}{$\gamma^*_{\text{DP}}$}
& \multicolumn{1}{c}{Worst $\apx_k$ (CE / Prophet)}
& \multicolumn{1}{c}{$\gamma^*_{\text{CE}}$}
& \multicolumn{1}{c}{Worst (CE / DP)}
& \multicolumn{1}{c}{$\gamma^*_{\text{CE/DP}}$} \\
\midrule
1   & 0.776 & 0.685 & 0.501 & 0.999 & 0.504 & 0.999 \\
2   & 0.853 & 0.630 & 0.667 & 0.999 & 0.668 & 0.999 \\
3   & 0.887 & 0.605 & 0.750 & 0.999 & 0.751 & 0.999 \\
5   & 0.921 & 0.585 & 0.833 & 0.999 & 0.834 & 0.999 \\
10  & 0.953 & 0.560 & 0.909 & 0.999 & 0.909 & 0.999 \\
20  & 0.972 & 0.550 & 0.952 & 0.999 & 0.953 & 0.999 \\
50  & 0.987 & 0.535 & 0.980 & 0.935 & 0.980 & 0.999 \\
100 & 0.993 & 0.530 & 0.990 & 0.830 & 0.990 & 0.999 \\
200 & 0.996 & 0.525 & 0.995 & 0.765 & 0.9995 & 0.999 \\
\bottomrule
\end{tabular}}
\label{table:worstcase}
\end{table}

\subsection{Finer Regret Comparison Between Optimal Dynamic Program and the CE heuristic}\label{subsec:finer-dp-ce}

Figure~\ref{fig:heatmaps_pareto} compares the additive gap $\vdp(n,k)-\vce(n,k)$ under a
Pareto distribution with $\gamma=0.7$ and joint scaling $k(n)=\lfloor n^\alpha\rfloor$ for
$\alpha\in\{0.4,0.6,0.8\}$. The left panel plots $\vdp(n,k(n))-\vce(n,k(n))$ on a log--log
scale and shows clear divergence as $n$ grows for every $\alpha<1$, with faster growth when the budget is sparser
(smaller $\alpha$). Since Pareto tails admit the closed-form identity
$\int_t^\infty \bar F(u)\,du=\frac{\gamma}{1-\gamma}t^{1-1/\gamma}$ for $t\ge1$, both $\vdp$ and
$\vce$ are computed exactly from their recursions, so the observed divergence is not an artifact of
asymptotic approximations.

The right panel reports the scaled quantity
\[
\frac{\vdp(n,k(n))-\vce(n,k(n))}{(n/k(n))^\gamma},
\]
which increases for small $n$ and then stabilizes across all $\alpha$. This behavior supports the power-law growth
\[
\vdp(n,k)-\vce(n,k)\ =\Theta \left((n/k)^\gamma\right) \qquad (k=o(n))
\]
in heavy-tailed settings, and in particular implies that the CE heuristic's additional regret diverges whenever
$k=o(n)$. In contrast to the ratio comparisons in
Figure~\ref{fig:heatmaps}--Table~\ref{tab:worst_case_ratios}, which approach $1$ quickly with $k$, this experiment shows
that a near-optimal competitive ratio does not preclude a growing additive loss when $n$ and $k$ scale jointly, which further validates Proposition~\ref{prop:regret-ce-dp-pareto}.

\begin{figure}[htbp]
  \centering
  \begin{subfigure}[b]{0.48\textwidth}
    \centering
    \includegraphics[height=5cm, keepaspectratio]{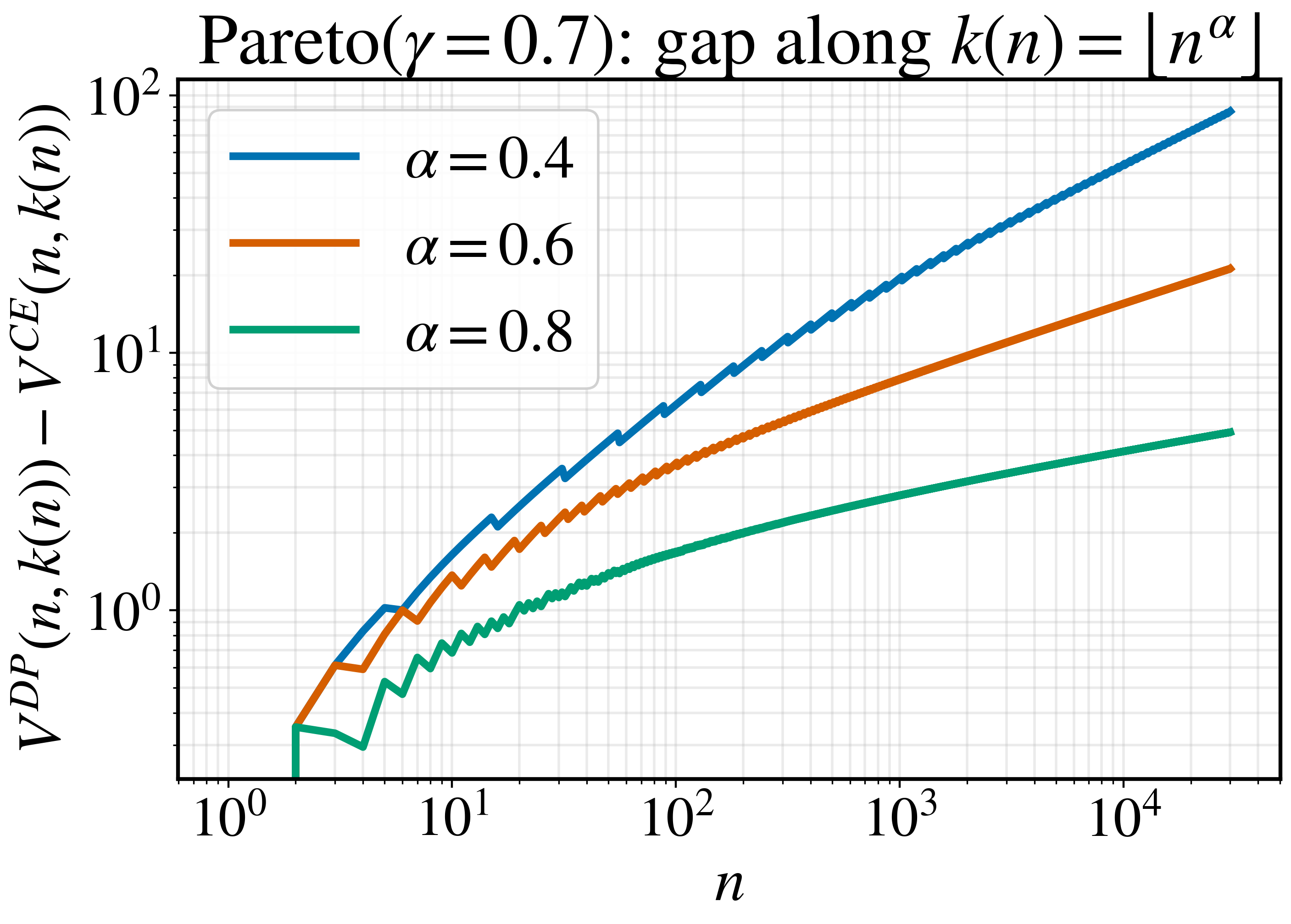}
    \caption{$\vdp-\vce$}
  \end{subfigure}
  \hfill
  \begin{subfigure}[b]{0.48\textwidth}
    \centering
    \includegraphics[height=5cm, keepaspectratio]{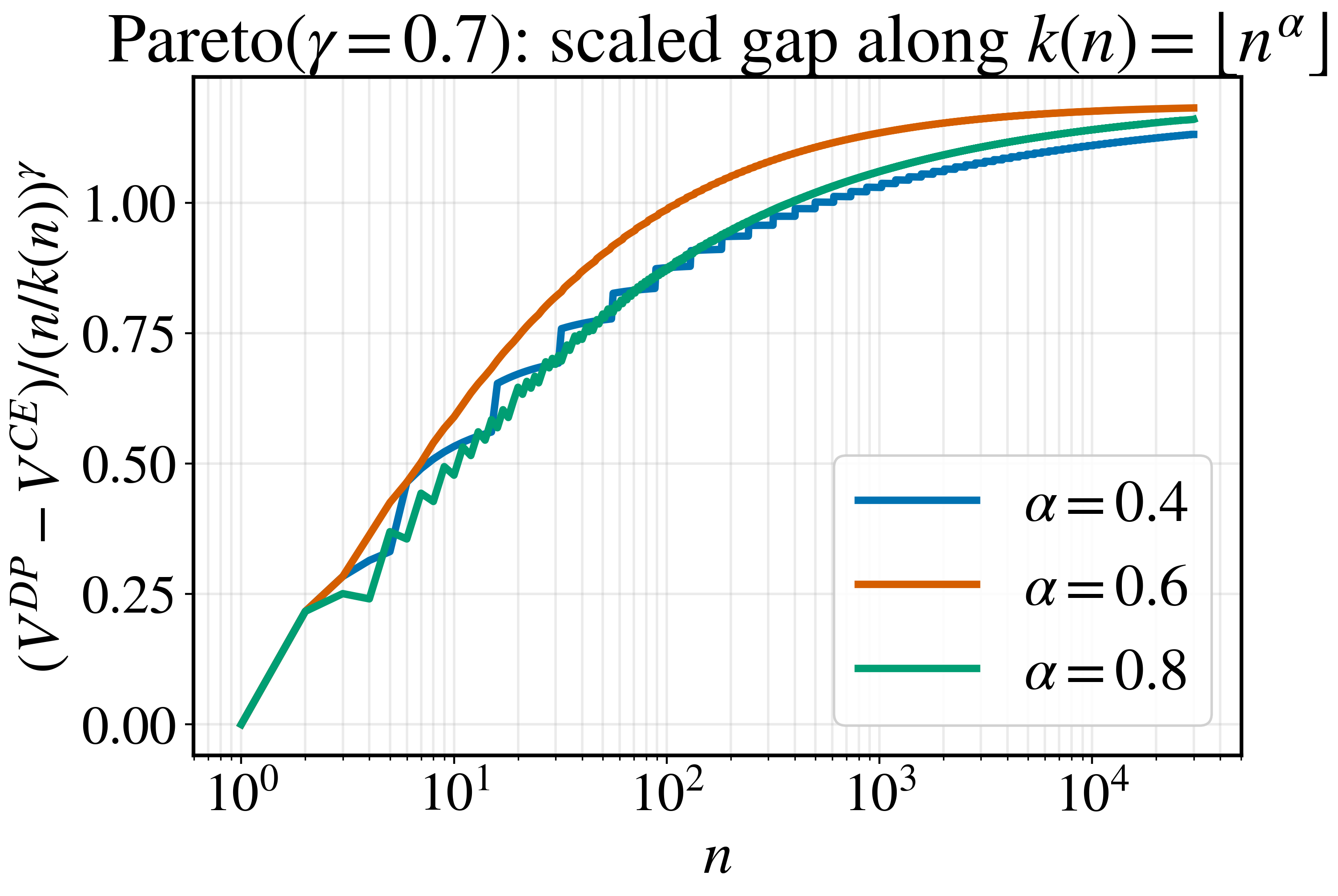}
    \caption{ ($\vdp-\vce$) / $(n/k(n))^\gamma$}
  \end{subfigure}
  \caption{Finer Comparison between DP and CE of Pareto Distribution}
  \label{fig:heatmaps_pareto}
\end{figure}

\section{Key Ingredients in the Proofs of Main Results}
\label{sec:proof}
In this section, we provide proofs for the key results, namely, characterizations of the solutions to the 2-dimensional recursions capturing the value functions $\vdp(n,k)$ and $\vce(n,k),$ for the case $\gamma > 0.$ Their application in proving the main results in the paper, and all other proofs not furnished here are presented in the appendix. 

When $F$ satisfies extreme value condition with $\gamma \neq 0$, the rate at which the distribution tail $\bar{F}(x) = 1 - F(x)$ decays to zero is characterized in terms of a broad class of functions known as \textit{regularly varying functions}. Given their central role in our analysis, we recall their definition in Definition~\ref{defn:RV} below.

\begin{definition}[Regular variation, see \cite{feller_1971} and \cite{Bingham_Goldie_Teugels_1987}]
\label{defn:RV}
A measurable function \(f:(0,\infty)\to(0,\infty)\) is said to be
regularly varying with index \(\gamma\in\mathbb R\) if, for every
\(c>0\),
\[
    \lim_{t\to\infty}\frac{f(ct)}{f(t)}=c^\gamma.
\]
We denote this by \(f\in \mathrm{\mathcal{RV}}_\gamma\).
\end{definition}

\subsection{Optimal Value Function Characterization}
\begin{lemma}
    For integers $n > 1$ and $k \ge 1$, the function $\vdp(n,k)$ satisfies
    \[
    \vdp(n,k)
    =
    \vdp(n-1,k)
    +
    \int_{\tau_{n,k}}^{x^*}
    \bar F(u)du,
    \]
    where $\tau_{n,k} := \vdp(n-1,k) - \vdp(n-1,k-1)$, and $x^*$ is the right end-point (possibly infinite) of $F$.
    \label{lem:dp-recursive}
\end{lemma}

\begin{proposition}
    Let $F$ be a distribution over $\R^+$ that satisfies the extreme value condition with $\gamma \in (0,1).$ Then for any $k \geq 1,$ the optimal value function $\vdp(n,k)$ satisfies 
    \begin{align*}
        \vdp(n,k) = \frac{v_k}{(1-\gamma)^{\gamma}} F^{\leftarrow}\left(1- \frac{1}{n} \right) [1 + o(1)],
    \end{align*}
    as $n \rightarrow \infty.$ Here $\{v_k\}_{k \geq 1}$ is obtained recursively as specified in the statement of Theorem \ref{thm:acr-dp}.
    \label{prop:dp-value-gamma-ge-0}
\end{proposition}

\begin{proposition}
    Let $F$ be a distribution over $\R^+$ that satisfies the extreme value condition with $\gamma < 0$ at a finite end-point $x^\ast.$  Then for any $k \geq 1,$ the optimal value function $\vdp(n,k)$ satisfies 
    \begin{align*}
        \vdp(n,k) = kx^\ast - \frac{v_k}{(1-\gamma)^{\gamma}} \left\{ x^\ast -  F^{\leftarrow}\left(1- \frac{1}{n} \right) \right\}[1 + o(1)],
    \end{align*}
    as $n \rightarrow \infty.$ Here $\{v_k\}_{k \geq 1}$ is obtained recursively as specified in the statement of Theorem \ref{thm:acr-dp}.
    \label{prop:dp-value-gamma-le-0}
\end{proposition}

\begin{proposition}
    Let $F$ be a distribution over $\R^+$ that satisfies the extreme value condition with $\gamma = 0.$  Then for any $k \geq 1,$ the optimal value function $\vdp(n,k)$ satisfies 
    \begin{align*}
        \vdp(n,k) = k  F^{\leftarrow}\left(1- \frac{1}{n} \right)    [1 + o(1)], \quad \text{ as } n \rightarrow \infty.
    \end{align*}
    \label{prop:dp-value-gamma-eq-0}
\end{proposition}

\begin{proof}[Proof of Proposition~\ref{prop:dp-value-gamma-ge-0}]
Fix \(k\geq 1\). Since \(\gamma\in(0,1)\), we have \(x^*=\infty\), and
Lemma~\ref{lem:dp-recursive} gives
\begin{equation}
\label{eq:dp-recursion-frechet}
\vdp(n,k)
=
\vdp(n-1,k)
+
R(\tau_{n,k}),
\qquad
R(t):=\int_t^\infty \overline F(u)\,du,
\end{equation}
where
\[
\tau_{n,k}
:=
\vdp(n-1,k)
-
\vdp(n-1,k-1).
\]

Define
\[
U(n):=F^{\leftarrow}\left(1-\frac1n\right),
\qquad
g_k(n):=\frac{\vdp(n,k)}{U(n)}.
\]
Since \(F\in\mathcal D_\gamma\) with \(\gamma\in(0,1)\), and by \cite{deHaan2006extreme} Corollary 1.2.10 and Theorem 1.2.1, respectively,
\[
U\in\mathcal{RV}_\gamma,
\qquad
\overline F\in\mathcal{RV}_{-1/\gamma}.
\]
By Karamata's theorem (see \cite{deHaan2006extreme} Theorem B.1.5 by setting $f(t)$ to be $\bar F(t)$),
\[
R(t)
\sim
\frac{\gamma}{1-\gamma}\,t\overline F(t),
\qquad t\to\infty.
\]
Since $\bar F \in \mathcal{RV}_{-1/\gamma}$,
\[
\frac{\bar F(U(n) y)}{\bar F(U(n))} \to y^{-1/\gamma},
\]
Consequently, we have
\begin{equation}
\label{eq:tail-integral-scaling}
\frac{nR(U(n)y)}{U(n)}
\longrightarrow
\frac{\gamma}{1-\gamma}y^{-(1-\gamma)/\gamma},
\end{equation}
locally uniformly for \(y\in(0,\infty)\).

We can now apply Lemma~\ref{lem:dp-asymptotic-recursion} in the appendix inductively in
\(k\) because function $R$ is nonnegative and nonincreasing. Since \(g_0(n)\equiv0\), the lemma implies that, for every fixed
\(k\geq1\),
\[
g_k(n)\longrightarrow g_k,
\]
where \(g_0=0\) and \(g_k>g_{k-1}\) is the unique solution of
\begin{equation}
\label{eq:gk-limit-recursion}
g_k
=
\frac{1}{1-\gamma}
\bigl(g_k-g_{k-1}\bigr)^{-(1-\gamma)/\gamma}.
\end{equation}

Now define
\[
v_k:=(1-\gamma)^\gamma g_k.
\]
Then \(v_0=0\), \(v_1=1\), and, setting
\[
z_k:=v_k-v_{k-1}>0,
\]
equation \eqref{eq:gk-limit-recursion} is equivalent to
\[
z_k^{1/\gamma}
+
v_{k-1}z_k^{1/\gamma-1}
-
1
=
0,
\]
which is exactly the recursion stated in Theorem~\ref{thm:acr-dp}.

Therefore,
\[
\vdp(n,k)
=
g_kU(n)[1+o(1)]
=
\frac{v_k}{(1-\gamma)^\gamma}
F^{\leftarrow}\left(1-\frac1n\right)[1+o(1)].
\]
\end{proof}

\subsection{Value Function of the CE Heuristic}

\begin{lemma}
    For integers $n > 1$ and $k \ge 1$, the function $\vce(n,k)$ satisfies
    \[
    \vce(n,k) = \frac{k}{n} \vce(n-1,k-1) + \int_{F^{\leftarrow}(1-\frac{k}{n})}^{x^*} \bar F(u) du + F^{\leftarrow}(1-\frac{k}{n}) \frac{k}{n}+(1-\frac{k}{n}) \vce(n-1,k),
    \]
    where $x^*$ is the right end-point(possibly infinite) of $F$.
    \label{lem:ce-recursive}
\end{lemma}

\begin{proposition}
    Let $F$ be a distribution over $\R^+$ that satisfies the extreme value condition with $\gamma \in (0,1).$ Then for any $k \geq 1,$ the CE heuristic's value function  $\vce(n,k)$ satisfies 
    \begin{align*}
        \vce(n,k) =  w_k F^{\leftarrow}\left(1- \frac{1}{n} \right) [1 + o(1)],
    \end{align*}
    as $n \rightarrow \infty.$ Here the sequence $\{ w_k\}_{k \geq 1}$ is obtained recursively from $w_1 = (1-\gamma^2)^{-1},$ and 
    \begin{align}
     w_k = \frac{k}{k+\gamma}  w_{k-1} + \frac{1}{k+\gamma} \frac{k^{1-\gamma}}{1-\gamma} \quad \text{ for any } k > 1.
    \label{eq:ck-recursion}
    \end{align}
    \label{prop:ce-value-gamma-ge-0}
\end{proposition}

\begin{proposition}
    Let $F$ be a distribution over $\R^+$ that satisfies the extreme value condition with $\gamma < 0$  at a finite end-point $x^\ast$. Then for any $k \geq 1,$ the CE heuristic's value function  $\vce(n,k)$ satisfies 
    \begin{align*}
        \vce(n,k) \longrightarrow kx^\ast.
    \end{align*}
    Moreover, if $-1<\gamma <0$, then
    \begin{align*}
        \vce(n,k) = kx^\ast -  w_k \left\{ x^\ast -  F^{\leftarrow}\left(1- \frac{1}{n} \right) \right\} [1 + o(1)],
    \end{align*}
    as $n \rightarrow \infty.$ Here the sequence $\{ w_k\}_{k \geq 1}$ is obtained recursively from $w_1 =(1-\gamma^2)^{-1},$ and  \eqref{eq:ck-recursion}. 
    \label{prop:ce-value-gamma-le-0}
\end{proposition}

\begin{proposition}
    Let $F$ be a distribution over $\R^+$ that satisfies the extreme value condition with $\gamma=0.$ Then for any $k \geq 1,$ the CE heuristic's value function  $\vce(n,k)$ satisfies 
    \begin{align*}
        \vce(n,k) = k F^{\leftarrow}\left(1- \frac{1}{n} \right) [1 + o(1)], \quad \text{ as } n \rightarrow \infty.
    \end{align*}
    \label{prop:ce-value-gamma-eq-0}
\end{proposition}

\begin{proof}[Proof of Proposition~\ref{prop:ce-value-gamma-ge-0}]
Fix \(k\geq1\). Since \(\gamma\in(0,1)\), we have \(x^*=\infty\).
For notational convenience, set
\[
    q_{n,k}:=F^{\leftarrow}\!\left(1-\frac{k}{n}\right)
    =U\!\left(\frac{n}{k}\right).
\]
Lemma~\ref{lem:ce-recursive} gives, for \(n>1\) and \(k\geq1\),
\begin{equation}
\label{eq:ce-rec}
\begin{aligned}
\vce(n,k)
&=
\frac{k}{n}\vce(n-1,k-1)
+\int_{q_{n,k}}^\infty \bar F(u)\,du \\
&\quad
+\frac{k}{n}q_{n,k}
+\left(1-\frac{k}{n}\right)\vce(n-1,k).
\end{aligned}
\end{equation}

Since \(F\in\mathcal D_\gamma\) with \(\gamma\in(0,1)\),
\(U\in\mathcal{RV}_\gamma\). By the Smooth Variation Theorem
\cite[Theorem~1.8.2]{Bingham_Goldie_Teugels_1987}, there exists an
eventually continuously differentiable function
\(\widetilde U\sim U\) such that
\[
    \frac{n\widetilde U'(n)}{\widetilde U(n)}
    \longrightarrow\gamma.
\]
Consequently,
\begin{equation}
\label{eq:Un-ratio-ce}
    \frac{\widetilde U(n-1)}{\widetilde U(n)}
    =
    1-\frac{\gamma}{n}
    +o\!\left(\frac1n\right).
\end{equation}

Define the normalized values
\[
    \widehat g_j(n)
    :=
    \frac{\vce(n,j)}{\widetilde U(n)},
    \qquad j\geq0,
\]
with \(\widehat g_0(n)\equiv0\). Since
\(\vce(n,j)\leq\mu_{n,j}=O(U(n))\) for every fixed \(j\), and
\(\widetilde U(n)\sim U(n)\), each sequence
\((\widehat g_j(n))_{n\geq1}\) is bounded.

Moreover, regular variation and \(\widetilde U(n)\sim U(n)\) give
\begin{equation}
\label{eq:q-ratio-ce}
    \frac{q_{n,k}}{\widetilde U(n)}
    =
    \frac{U(n/k)}{\widetilde U(n)}
    =
    k^{-\gamma}[1+o(1)].
\end{equation}

Since \(\bar F\in\mathcal{RV}_{-1/\gamma}\), Karamata's theorem gives
\begin{equation}
\label{eq:karamata-ce}
    \int_t^\infty\bar F(u)\,du
    =
    \frac{\gamma}{1-\gamma}
    t\bar F(t)[1+o(1)],
    \qquad t\to\infty.
\end{equation}
Combining \eqref{eq:q-ratio-ce} and \eqref{eq:karamata-ce}, we obtain
\begin{equation}
\label{eq:tailint-ce}
\begin{aligned}
\frac{1}{\widetilde U(n)}
\int_{q_{n,k}}^\infty\bar F(u)\,du
&=
\frac{\gamma}{1-\gamma}
\frac{q_{n,k}}{\widetilde U(n)}
\bar F(q_{n,k})[1+o(1)] \\
&=
\frac{1}{n}\frac{\gamma}{1-\gamma}
k^{1-\gamma}[1+o(1)].
\end{aligned}
\end{equation}

Dividing \eqref{eq:ce-rec} by \(\widetilde U(n)\), and using
\eqref{eq:Un-ratio-ce}--\eqref{eq:tailint-ce}, gives
\begin{align}
\widehat g_k(n)
&=
\frac{k}{n}
\frac{\widetilde U(n-1)}{\widetilde U(n)}
\widehat g_{k-1}(n-1)
+
\left(1-\frac{k}{n}\right)
\frac{\widetilde U(n-1)}{\widetilde U(n)}
\widehat g_k(n-1)
\notag\\
&\quad
+\frac{k}{n}\frac{q_{n,k}}{\widetilde U(n)}
+
\frac{1}{\widetilde U(n)}
\int_{q_{n,k}}^\infty\bar F(u)\,du
\notag\\
&=
\left(1-\frac{k+\gamma}{n}\right)
\widehat g_k(n-1)
+
\frac{k}{n}\widehat g_{k-1}(n-1)
\notag\\
&\quad
+
\frac{1}{n}\frac{k^{1-\gamma}}{1-\gamma}
+
o\!\left(\frac1n\right).
\label{eq:gk-diff-ce}
\end{align}
Equivalently,
\begin{equation}
\label{eq:gk-increment-ce}
\widehat g_k(n)-\widehat g_k(n-1)
=
\frac1n
\left[
-(k+\gamma)\widehat g_k(n-1)
+k\widehat g_{k-1}(n-1)
+\frac{k^{1-\gamma}}{1-\gamma}
+o(1)
\right].
\end{equation}

We now proceed inductively in \(k\) to show the convergence of $\widehat g_k(n)$. The assertion is immediate for
\(k=0\). Suppose that
\[
    \widehat g_{k-1}(n)\longrightarrow w_{k-1}.
\]
Then \eqref{eq:gk-increment-ce} becomes
\[
\widehat g_k(n)-\widehat g_k(n-1)
=
\frac1n
\left[
-(k+\gamma)\widehat g_k(n-1)
+kw_{k-1}
+\frac{k^{1-\gamma}}{1-\gamma}
+o(1)
\right].
\]
Since \(k+\gamma>0\), Lemma~\ref{lem:linear-asymptotic-recursion}
implies that
\[
    \widehat g_k(n)\longrightarrow w_k,
\]
where
\[
    w_k
    =
    \frac{k}{k+\gamma}w_{k-1}
    +
    \frac{k^{1-\gamma}}
    {(k+\gamma)(1-\gamma)}.
\]
With \(w_0=0\), this gives
\[
    w_1
    =
    \frac{1}{(1+\gamma)(1-\gamma)}
    =
    \frac{1}{1-\gamma^2},
\]
and determines \(w_k\) recursively for every fixed \(k\geq1\).

Finally, since \(\widetilde U(n)\sim U(n)\),
\[
\begin{aligned}
\vce(n,k)
=
\widetilde U(n)\widehat g_k(n) 
&=
w_kU(n)[1+o(1)].
\end{aligned}
\]
\end{proof}

\subsection{The Prophet's Performance}

\begin{theorem}
    \label{thm:prophet}
    Let $X_1,\dots,X_n$ be i.i.d.\ with distribution $F\in D_\gamma$. Let
    $X_{1:n}\le\cdots\le X_{n:n}$ denote the order statistics and an integer $k\ge 1$.
    Define
    \[
    \mu_{n,k}\ :=\ \mathbb{E}\Big[\sum_{j=0}^{k-1} X_{n-j:n}\Big]
    \ =\ \sum_{r=1}^{k}\mathbb{E}[X_{n-r+1:n}],
    \]
    and let $x^*$ be the (possibly infinite) right endpoint of $F$.
    Then, for large enough $n$,
    \[
    \mu_{n,k}\ =
    \begin{cases}
    \displaystyle
    \frac{\Gamma(k+1-\gamma)}{(1-\gamma)\Gamma(k)}
    F^{\leftarrow}\!\Big(1-\frac{1}{n}\Big)[1+o(1)],
    & \gamma\in(0,1),
    \\[3mm]
    \displaystyle
    k\,F^{\leftarrow}\!\Big(1-\frac{1}{n}\Big)
 +\ a_n\,\sum_{r=1}^{k}\bigl(-\psi(r)\bigr) + o(a(n)),
    & \gamma=0,
    \\[3mm]
    \displaystyle
    k\,x^*
    \ -\
    \frac{\Gamma(k+1-\gamma)}{(1-\gamma)\Gamma(k)}
    \Big(x^*-F^{\leftarrow}\!\Big(1-\frac{1}{n}\Big)\Big)[1+o(1)],
    & \gamma<0,
    \end{cases}
    \]
    where $\psi(\cdot)$ is the digamma function, and $a_n$ is appropriately chosen, e.g.
    \[
    \psi(r) := \frac{\Gamma'(r)}{\Gamma(r)}, \qquad a_n\ = {F^{\leftarrow}(1-\frac{1}{en})-F^{\leftarrow}(1-\frac{1}{n})},
    \qquad x>1.
    \]

    % \[
    % \psi(r) = H_{r-1} - \gamma_{E}, \qquad H_m = \sum_{j=1}^m \frac{1}{j}, \qquad a_n\ = {F^{\leftarrow}(1-\frac{1}{en})-F^{\leftarrow}(1-\frac{1}{n})},
    % \qquad x>1,
    % \]
    % where $\gamma_E \approx 0.5772$ is the Euler–Mascheroni constant.

    % When k is large, 
    % \[
    % \mu_{n,k}\ =
    % \begin{cases}
    % \displaystyle
    % \frac{k^{1-\gamma}}{1-\gamma}\left(1
    % -
    % \frac{\gamma(1-\gamma)}{2}\frac{1}{k}
    % +
    % O\!\left(\frac{1}{k^2}\right) \right) F^{\leftarrow}\!\left(1-\frac{1}{n}\right),
    % & \gamma\in(0,1),
    % \\[3mm]
    % \displaystyle
    % \mu_{n,k} = kF^{\leftarrow}(1-\frac{1}{n}) + a(n)(-k \log k + k -\frac{1}{2} + o(1)),
    % & \gamma=0,
    % \\[3mm]
    % \displaystyle
    % kx^* - 
    % \frac{k^{1-\gamma}}{1-\gamma}\left(1
    % -
    % \frac{\gamma(1-\gamma)}{2}\frac{1}{k}
    % +
    % O\!\left(\frac{1}{k^2}\right) \right) \left(x^* - F^{\leftarrow}\!\left(1-\frac{1}{n}\right)\right),
    % & \gamma<0,
    % \end{cases}
    % \]
\end{theorem}

\section{Competition Complexity in Large Markets}\label{sec:comp-complexity}

This section recasts competitive-ratio guarantees into a ``market inflation'' perspective: In particular, how much larger must the
market be in order to offset a policy's multiplicative loss relative to the Prophet? % at the extreme-value scale.
We study this notion, addressed as competition complexity (see, eg., \cite{beyhaghi2019optimal,Brustle-competition,correa2025}), from an instance-dependent perspective in this section.  Throughout we fix the number of units $k\ge 1$.

\subsection{Definition}\label{subsec:cc-def}
For a policy $\pi\in\{\tt{DP},\tt{CE}\}$, write
$\acr^{\pi}_k(F)\in(0,1]$ for its asymptotic competitive ratio (as $n\to\infty$ with $k$ fixed), namely
\[
\acr^{\tt{DP}}_k(F)=\acr_k(F),\qquad
\acr^{\tt{CE}}_k(F)=\apx_k(F).
\]
From a posted-price mechanism perspective, the notion of competition complexity, defined below, inquires into how many more bidders are needed for the policy $\pi$ to match the reward of a Prophet working on a market with $n$ users. 
\begin{definition}[Competition complexity]\label{def:cc}
Define the large-market competition complexity of policy $\pi$ at budget $k$ by
\begin{align*}
    C^\pi_k(F) := \liminf_{n \rightarrow \infty} \inf_{m \geq n} \left\{ \frac{m}{n}: V^\pi(m,k) \geq \mu_{n,k}\right\}.
\end{align*}
\end{definition}

\noindent 
Due to the characterizations of $\vdp, \vce,$ and $\mu_{n,k}$ in Sections \ref{sec:proof}, we equivalently write  
\begin{equation}\label{eq:cc-def}
\mathcal{C}^{\pi}_k(F)
\;:=\;
\liminf_{n\to\infty}\;
\inf_{m\ge n}
\left\{
\frac{m}{n}
:\;
\acr^{\pi}_k(F)\,U(m)\ \ge\ U(n)
\right\},
\end{equation}
In turn, the inequality in \eqref{eq:cc-def} asks for the smallest asymptotic inflation factor $m/n$ such that the loss
$\acr^{\pi}_k(F)$ can be compensated by moving from quantile $U(n)$ to $U(m),$ where $U(n) := F^{\leftarrow}(1-1/n).$

\subsection{Closed-form characterization}\label{subsec:cc}

\begin{proposition}[Competition complexity for $\gamma\in(0,1)$]\label{prop:cc-frechet}
Let $F\in D_\gamma$ with $\gamma\in(0,1)$ and fix $k\ge 1$. Then
\begin{equation}\label{eq:cc-frechet}
\mathcal{C}^{\tt{DP}}_k(F)=\acr_k(F)^{-1/\gamma},
\qquad
\mathcal{C}^{\tt{CE}}_k(F)=\apx_k(F)^{-1/\gamma}.
\end{equation}
\end{proposition}

\begin{proof}
Since $F\in D_\gamma$ with $\gamma>0$, the quantile function $U$ is regularly varying with index $\gamma$; in
particular, for every fixed $c\ge 1$,
\begin{equation}\label{eq:U-rv}
\frac{U(\lfloor cn\rfloor)}{U(n)}\to c^\gamma
\qquad (n\to\infty).
\end{equation}
Fix $c\ge 1$ and set $m=\lfloor cn\rfloor$. Dividing the condition
$\acr^{\pi}_k(F)\,U(m)\ge U(n)$ by $U(n)$ and applying \eqref{eq:U-rv} yields
\[
\acr^{\pi}_k(F)\,\frac{U(m)}{U(n)}
\to
\acr^{\pi}_k(F)\,c^\gamma.
\]
Hence, for large $n$, the inequality holds whenever $\acr^{\pi}_k(F)\,c^\gamma\ge 1$, i.e.,
$c\ge (\acr^{\pi}_k(F))^{-1/\gamma}$. Taking the infimum over such $c$ and then $\liminf_{n\to\infty}$
in \eqref{eq:cc-def} gives
\[
\mathcal{C}^{\pi}_k(F)=(\acr^{\pi}_k(F))^{-1/\gamma}.
\]
Substituting $\acr^{\tt{DP}}_k(F)=\acr_k(F)$ and $\acr^{\tt{CE}}_k(F)=\apx_k(F)$ proves
\eqref{eq:cc-frechet}.
\end{proof}

% When $\gamma\le 0$, our main results yield $\acr_k(F)=\apx_k(F)=1$ for each fixed $k$, and hence the loss vanishes
% at the relevant extreme-value scale. Under \eqref{eq:cc-def}, choosing $m=n$ shows
% \[
% \mathcal{C}^{\tt{DP}}_k(F)=\mathcal{C}^{\tt{CE}}_k(F)=1.
% \]

\subsection{Large-$k$ expansion}\label{subsec:cc-large-k}

In the Fr\'echet regime $\gamma\in(0,1)$, Propositions~\ref{prop:dp-large-k} and~\ref{prop:ce-large-k} imply that
there exist constants $M_{\tt{DP}},M_{\tt{CE}}\in\R$ such that for all $k\ge 1$,
\[
\acr_k(F)
=
1-\frac{\gamma(1-\gamma)}{2}\frac{\log k}{k}+O\!\Bigl(\frac{1}{k}\Bigr),
\qquad
\apx_k(F)
=
1-\frac{\gamma(1-\gamma)}{2}\frac{\log k}{k}+O\!\Bigl(\frac{1}{k}\Bigr),
\]
with $O(1/k)$ uniform over $\gamma\in(0,1)$. Combining with \eqref{eq:cc-frechet} and the expansion
\[
(1-x)^{-1/\gamma}=1+\frac{x}{\gamma}+O(x^2)\qquad (x\to 0),
\]
yields
\begin{equation}\label{eq:cc-large-k}
\mathcal{C}^{\tt{DP}}_k(F)
=
1+\frac{1-\gamma}{2}\frac{\log k}{k}+O\!\Bigl(\frac{1}{k}\Bigr),
\qquad
\mathcal{C}^{\tt{CE}}_k(F)
=
1+\frac{1-\gamma}{2}\frac{\log k}{k}+O\!\Bigl(\frac{1}{k}\Bigr),
\qquad k\to\infty.
\end{equation}
Thus, in large budgets, both policies require only a $(1+O(\frac{\log k}{k}))$ market-size inflation to match the
oracle's extreme-value scale, and DP and CE agree at the leading $\frac{\log k}{k}$ order.

\section{Concluding Remarks}

First, adopting the large-market viewpoint ($n\to\infty$ with the instance distribution $F$ fixed) yields a sharp and  interpretable characterization of optimal online performance for the multiunit online selection problem: under the extreme value condition, the instance-dependent asymptotic competitive ratio $\acr_k(F)$ depends on $F$ only through its extreme value index $\gamma$. This EVT-based reduction both unifies and extends prior asymptotic analyses, and it enables explicit worst-case guarantees over broad distributional classes. In particular, it leads to a tight large-$k$ worst-case guarantee of the form
$1-\frac{\log k}{8k}[1+\varepsilon]$, which despite being asymptotic, provides one of the strongest currently available analytical expressions for  worst-case performance over distributions satisfying the extreme value condition.

Second, the same framework makes it possible to rigorously quantify the performance loss from restricting attention to simpler, more interpretable policies. While fixed-threshold policies admit clean guarantees and achieve worst-case ratios approaching $1-1/\sqrt{2\pi k}$, the CE heuristic provides a more adaptive alternative that is widely used beyond prophet-inequality settings. Our results show that the CE heuristic indeed improves with $k$ and ultimately matches the leading-order performance of the optimal dynamic program as $k$ becomes large, even without coupling $k$ and $n$ through a fluid scaling. However, a finer comparison reveals an important caveat: when $k=o(n)$ and $n\to\infty$, the CE heuristic can incur divergent regret relative to the optimal dynamic program for continuous reward distributions, in sharp contrast to the uniformly bounded regret phenomena established for finitely supported rewards. This divergence underscores that conclusions drawn under the fluid scaling assumption $k\propto n$ can be qualitatively misleading when applied to fixed instances where $k \ll n.$

Our results point to several natural next steps.
First, it would be interesting to extend the EVT-based instance-dependent perspective beyond the single $k$-unit constraint to more general online linear programs with multi-dimensional budgets. CE-style policies and primal--dual methods are widely used in these settings because they are simple and scalable, but we currently lack a sharp characterization of their instance-dependent performance comparable to the $\gamma$-indexed description developed here. Understanding whether an extreme-value summary of the instance still governs asymptotic performance, and what the right analogue of the tail index should be under multiple constraints with random coefficients, are open questions.

Second, an important practical extension is to drop the assumption that the reward distribution $F$ is known. In many applications the decision-maker must learn $F$ from data or on the fly, which introduces a learning component. A key challenge is to design adaptive algorithms that learn the relevant tail features (such as high quantiles or the extreme value index) while maintaining strong performance guarantees, especially in regimes where $k$ and $n$ are decoupled and fluid approximations are not reliable.
%\noindent
%\textbf{Acknowledgements.} 

\bibliographystyle{abbrv}
\bibliography{references}
\clearpage
\setcounter{page}{1}

%\appendix
%\numberwithin{equation}{section}
%\begin{center}
%  \textbf{Supplementary material for the paper ``''}
%\end{center}
%\section{}

\vspace{10pt}
%Appendix
\appendix
\begin{center}
    \textbf{Appendices}
\end{center}
\section{Proofs of results in Section \ref{sec:dp}} 

\subsection{Proof of Theorem~\ref{thm:acr-dp}}

\begin{proof}[Proof of Theorem~\ref{thm:acr-dp}]
Recall that the optimal asymptotic competitive ratio of the dynamic program is
\[
\acr_k(F):=\lim_{n\to\infty}\frac{\vdp(n,k)}{\mu_{n,k}},
\]
where $\mu_{n,k}$ is the prophet value defined in Theorem~\ref{thm:prophet}.

\subsubsection*{Step 1: $\gamma\in(0,1)$}
Fix $\gamma\in(0,1)$. By Proposition~\ref{prop:dp-value-gamma-ge-0}, for each fixed $k\ge1$,
\begin{equation}\label{eq:dp-asymp}
\vdp(n,k)
=
\frac{v_k}{(1-\gamma)^\gamma}\,
F^{\leftarrow}\!\Bigl(1-\frac{1}{n}\Bigr)\,[1+o(1)],
\qquad n\to\infty.
\end{equation}
On the other hand, by Theorem~\ref{thm:prophet},
\begin{equation}\label{eq:oracle-asymp}
\mu_{n,k}
=
\frac{\Gamma(k+1-\gamma)}{(1-\gamma)\Gamma(k)}\,
F^{\leftarrow}\!\Bigl(1-\frac{1}{n}\Bigr)\,[1+o(1)],
\qquad n\to\infty.
\end{equation}
Dividing \eqref{eq:dp-asymp} by \eqref{eq:oracle-asymp} and letting $n\to\infty$ yields
\[
\acr_k(F)
=
\lim_{n\to\infty}\frac{\vdp(n,k)}{\mu_{n,k}}
=
\frac{\frac{v_k}{(1-\gamma)^\gamma}}{\frac{\Gamma(k+1-\gamma)}{(1-\gamma)\Gamma(k)}}
=
(1-\gamma)^{1-\gamma}\,\frac{v_k\,\Gamma(k)}{\Gamma(k+1-\gamma)}.
\]
This proves the first line of \eqref{eq:acr-dp}.

\subsubsection*{Step 2: $\gamma\notin(0,1)$}
If $\gamma=0$, then Theorem~\ref{thm:prophet} gives
\[
\mu_{n,k}
=
k\,F^{\leftarrow}\!\Bigl(1-\frac{1}{n}\Bigr)
+
a_n\sum_{r=1}^k\bigl(-\psi(r)\bigr),
\qquad n\to\infty,
\]
where $a_n=o(F^{\leftarrow}(1-\frac{1}{n}))$. Hence
\[
\mu_{n,k}
=
k\,F^{\leftarrow}\!\Bigl(1-\frac{1}{n}\Bigr)\,[1+o(1)].
\]
The $\gamma=0$ Proposition~\ref{prop:dp-value-gamma-eq-0} yields
\[
\vdp(n,k)
=
k\,F^{\leftarrow}\!\Bigl(1-\frac{1}{n}\Bigr)\,[1+o(1)].
\]
Therefore $\acr_k(F)=\lim_{n\to\infty}\vdp(n,k)/\mu_{n,k}=1$.

If $\gamma<0$, let $x^\ast<\infty$ be the right endpoint of $F$. Theorem~\ref{thm:prophet} gives
\[
\mu_{n,k}
=
k\,x^\ast
-
\frac{\Gamma(k+1-\gamma)}{(1-\gamma)\Gamma(k)}
\Bigl(x^\ast-F^{\leftarrow}\!\Bigl(1-\frac{1}{n}\Bigr)\Bigr)\,[1+o(1)].
\]
The $\gamma<0$ Proposition~\ref{prop:dp-value-gamma-le-0} gives the same leading term $kx^\ast$ for
$\vdp(n,k)$, and we have $F^{\leftarrow}(1-\frac{1}{n}) \to x^\ast$ when $n\to \infty$. Dividing again yields
$\vdp(n,k)/\mu_{n,k}=1+o(1)$ and hence $\acr_k(F)=1$.

This proves the second line of \eqref{eq:acr-dp}.

\end{proof}

\subsection{Proof of Proposition~\ref{prop:dp-large-k}}

\begin{lemma}\label{lem:Gamma-ratio-exp}
Fix $\gamma\in(0,1)$ and set $a:=1-\gamma\in(0,1)$. Define
\[
B_k:=k^{a}\frac{\Gamma(k)}{\Gamma(k+a)}\qquad (k\ge1).
\]
Then
\[
B_k=1+\frac{a(1-a)}{2k}+O\!\left(\frac1{k^2}\right)
=1+\frac{\gamma(1-\gamma)}{2k}+O\!\left(\frac1{k^2}\right),
\]
and in particular there exists $C_1<\infty$ such that $|B_k-1|\le C_1/k$ for all $k\ge1$.
\end{lemma}

\begin{proof}
By Stirling's formula with remainder for $\log\Gamma$,
\[
\log\Gamma(z) = \Bigl(z-\tfrac12\Bigr)\log z - z + \tfrac12\log(2\pi) + \frac{1}{12z}
+O\!\left(\frac{1}{z^3}\right)
\qquad (z\to\infty).
\]
Apply this to $z=k$ and $z=k+a$ and subtract to obtain
\[
\log\frac{\Gamma(k)}{\Gamma(k+a)}
= -a\log k + \frac{a(1-a)}{2k}+O\!\left(\frac1{k^2}\right),
\]
where we utilize 
\[
\log (1 + \frac{a}{k}) = \frac{a}{k} -\frac{a^2}{2k^2} +O\left( \frac{1}{k^3}\right). \qquad (k\to\infty).
\]
Exponentiating yields
\begin{align*}
\frac{\Gamma(k)}{\Gamma(k+a)} &= k^{-a}\exp{\left(\frac{a(1-a)}{2k}+O\!\left(\frac1{k^2}\right)\right)}\\
&= k^{-a} \left(1+\frac{a(1-a)}{2k}+O\left(\frac1{k^2}\right)\right) \qquad (k\to\infty)
\end{align*}
% \[
% k^{a}\frac{\Gamma(k)}{\Gamma(k+a)} = k^{-a}\left(1+\frac{a(1-a)}{2k}+O\!\left(\frac1{k^2}\right)\right),
% \]
and multiplying by $k^a$ gives the claim. The bound $|B_k-1|\le C_1/k$ follows immediately.
\end{proof}

\begin{lemma}\label{lem:sk-asymp}
Fix $\gamma\in(0,1)$ and let $\{v_k\}_{k\ge1}$ be the sequence in Theorem~\ref{thm:acr-dp}. Define
\[
s_k:=(1-\gamma)\,v_k^{1/(1-\gamma)}\qquad (k\ge1).
\]
Then, as $k\to\infty$,
\[
s_k = k-\frac{\gamma}{2}\log k + O(1),
\qquad\text{and hence}\qquad
\frac{s_k}{k}=1-\frac{\gamma}{2}\frac{\log k}{k}+O\!\left(\frac1k\right).
\]
\end{lemma}

\begin{proof}
Let $\Delta_k:=v_k-v_{k-1}$. The defining equation in Theorem~\ref{thm:acr-dp} is equivalent to
\[
\Delta_k^{1/\gamma-1}(v_{k-1}+\Delta_k)=1
\quad\Longleftrightarrow\quad
\Delta_k^{1/\gamma-1}v_k=1
\quad\Longleftrightarrow\quad
\Delta_k=v_k^{-\gamma/(1-\gamma)}.
\]
Therefore,
\begin{equation}\label{eq:vk-diff-lemma}
v_k-v_{k-1}=v_k^{-\gamma/(1-\gamma)}\qquad (k\ge2).
\end{equation}
Dividing \eqref{eq:vk-diff-lemma} by $v_k$ gives
\[
\frac{v_k-v_{k-1}}{v_k}=v_k^{-1/(1-\gamma)}=\frac{1-\gamma}{s_k},
\]
so
\[
v_{k-1}=v_k\Bigl(1-\frac{1-\gamma}{s_k}\Bigr).
\]
Raising to the power $1/(1-\gamma)$ and multiplying by $(1-\gamma)$ yields the exact identity
\begin{equation}\label{eq:sk-exact-lemma}
s_{k-1}=s_k\Bigl(1-\frac{1-\gamma}{s_k}\Bigr)^{1/(1-\gamma)}.
\end{equation}

Next we show $s_k \to \infty$ as $k\to \infty$. In fact, if $v_k \le L$ for some $L\in \R$ for every $k$, then from the recursion below,
\[
v_k - v_{k-1} = v_k^{-\gamma/(1-\gamma)} \ge L^{-\gamma/(1-\gamma)},
\]
which clearly implies $v_k$ will be unbounded from above, so it leads to the contradiction.

Therefore, as $k \to \infty$, $v_k \to \infty$, and thus $s_k = (1-\gamma) v_k^{1/(1-\gamma)}\to \infty$ and it is non-decreasing.

Write $u:=(1-\gamma)/s_k$. Since $s_k\uparrow\infty$, we have $u\downarrow0$.
For $|u|\le 1/2$, a third-order Taylor expansion of $(1-u)^{1/(1-\gamma)}$ around $u=0$ gives
\begin{equation}\label{eq:taylor-lemma}
(1-u)^{1/(1-\gamma)}
=
1-\frac{u}{1-\gamma}+\frac{\gamma}{2(1-\gamma)^2}u^2+R(u),
\qquad |R(u)|\le C\,u^3,
\end{equation}
for some constant $C<\infty$ (depending only on $\gamma$). Choose $k_0$ such that
$u\le 1/2$ for all $k\ge k_0$. Substituting \eqref{eq:taylor-lemma} with $u=(1-\gamma)/s_k$
into \eqref{eq:sk-exact-lemma} yields, for all $k\ge k_0$,
\[
s_{k-1}
= s_k\left(1-\frac{1}{s_k}+\frac{\gamma}{2s_k^2}+O\!\left(\frac{1}{s_k^3}\right)\right)
= s_k-1+\frac{\gamma}{2s_k}+O\!\left(\frac{1}{s_k^2}\right),
\]
and hence
\begin{equation}\label{eq:sk-inc-lemma}
s_k-s_{k-1}=1-\frac{\gamma}{2s_k}+O\!\left(\frac{1}{s_k^2}\right)\qquad (k\ge k_0).
\end{equation}
From \eqref{eq:sk-inc-lemma}, we have $s_k-s_{k-1}=1+o(1)$ since $s_k \uparrow \infty$ as $k \to \infty$.

In particular, for sufficiently larg $k\ge k_0$, $s_k - s_{k-1} \ge \frac{1}{2}$, and then $\exists c_0>0$ such that 
\[
s_k \ge c_0k,
\]
and hence 
\begin{align*}
0\ge s_k - s_{k-1} -1 &= -\frac{\gamma}{2s_k} + O(s_k^{-2})\\
&\ge \frac{-\gamma}{2ck} + O(s_k^{-2}),
\end{align*}
for sufficiently large $k$, and then $\exists c_1>0$, such that 
\[
|s_k - s_{k-1} - 1| \le \frac{c_1}{k}.
\]
Therefore, there exists $k_1\ge k_0$ such that for $\forall k\ge k_1+1$,
\[
s_k \ge k - c_1\sum_{j=k_1+1}^{k} \frac1{j} \ge k - c_1 \log k.
\]
Further, there exists $k_2 \ge k_1$ and $c_2>0$ such that 
\[
\frac{1}{s_k} - \frac{1}{k} \le \frac{1}{k-c_1 \log k} - \frac{1}{k} \le c_2 \frac{\log k}{k^2}.
\]
Also notice 
\[
\sum_{j=1}^\infty \frac{\log j}{j^2} < \infty.
\]
It yields
\[
\sum_{j=k_2+1}^k\frac{1}{s_j} = \sum_{j=k_2+1}^k\frac{1}{j} + O(1) = \log k + O(1),
\]
along with
\[
\sum_{j=k_2+1}^\infty \frac{1}{s_j^2}\ \le\ \frac{1}{c_0^2}\sum_{j=k_2+1}^\infty \frac{1}{j^2}\ <\ \infty.
\]
Therefore, summing \eqref{eq:sk-inc-lemma} from $k_2+1$ to $k$ gives
\begin{align}\label{eq:sk-sum-lemma}
s_k
&= s_{k_2}+(k-k_2)-\frac{\gamma}{2}\sum_{j=k_2+1}^k\frac{1}{s_j}
+O\!\left(\sum_{j=k_2+1}^k\frac{1}{s_j^2}\right).\\
&= k-\frac{\gamma}{2}\log k+O(1)
\end{align}

\end{proof}

\begin{lemma}\label{lem:Ak-exp}
With $A_k:=\dfrac{(1-\gamma)^{1-\gamma}v_k}{k^{1-\gamma}}$, we have
\[
A_k
=\left(\frac{s_k}{k}\right)^{1-\gamma}
=1-\frac{\gamma(1-\gamma)}{2}\frac{\log k}{k}+O\!\left(\frac{1}{k}\right).
\]
\end{lemma}

\begin{proof}
By definition, $A_k=(1-\gamma)^{1-\gamma}v_k/k^{1-\gamma}=(s_k/k)^{1-\gamma}$.
Lemma~\ref{lem:sk-asymp} gives $\frac{s_k}{k}=1+z_k$ with
$z_k=-(\gamma/2)\,(\log k)/k+O(1/k)$. Since $z_k\to0$, the expansion
$(1+z)^{1-\gamma}=1+(1-\gamma)z+O(z^2)$ yields
\[
A_k
=1+(1-\gamma)z_k+O(z_k^2)
=1-\frac{\gamma(1-\gamma)}{2}\frac{\log k}{k}+O\!\left(\frac{1}{k}\right),
\]
because $z_k^2=O((\log k)^2/k^2)=O(1/k)$.
\end{proof}

\begin{proof}[Proof of Proposition~\ref{prop:dp-large-k}]
Fix $\gamma\in(0,1)$. By Theorem~\ref{thm:acr-dp},
\[
\acr_k(F)=A_k\,B_k,
\qquad
A_k:=\frac{(1-\gamma)^{1-\gamma}v_k}{k^{1-\gamma}},
\qquad
B_k:=k^{1-\gamma}\frac{\Gamma(k)}{\Gamma(k+1-\gamma)}.
\]
By Lemma~\ref{lem:Ak-exp} and Lemma~\ref{lem:Gamma-ratio-exp},
\[
A_k
=
1-\frac{\gamma(1-\gamma)}{2}\frac{\log k}{k}+O\!\left(\frac{1}{k}\right),
\qquad
B_k=1+O\!\left(\frac{1}{k}\right).
\]
Multiplying these two expansions gives
\[
\acr_k(F)
=
1-\frac{\gamma(1-\gamma)}{2}\frac{\log k}{k}+O\!\left(\frac{1}{k}\right),
\qquad k\to\infty.
\]
Hence there exist $k_\ast$ and $M_1<\infty$ such that, for all $k\ge k_\ast$,
\[
\left|\acr_k(F)-\left(1-\frac{\gamma(1-\gamma)}{2}\frac{\log k}{k}\right)\right|
\le \frac{M_1}{k}.
\]
For the finitely many indices $1\le k<k_\ast$, define
\[
M_2
:=\max_{1\le k<k_\ast}
k\left|\acr_k(F)-\left(1-\frac{\gamma(1-\gamma)}{2}\frac{\log k}{k}\right)\right|.
\]
Then $M_2<\infty$ and the same inequality holds with $M_2$ for all $1\le k<k_\ast$.
Taking $M:=\max\{M_1,M_2\}$ yields \eqref{eq:dp-large-k} for every $k\ge1$.
\end{proof}

\subsection{Proof of Corollary~\ref{cor:dp-large-k}}

\begin{proof}
    Fix $\varepsilon>0$. If $\gamma\notin(0,1)$, then by Theorem~\ref{thm:acr-dp} we have $\acr_k(F)=1$ for every
    $k$, and the desired inequality holds trivially for all $k$.
    
    Now assume $\gamma\in(0,1)$. By Proposition~\ref{prop:dp-large-k}, there exists $M<\infty$ such that for all
    $k\ge1$,
    \begin{equation}\label{eq:acr-lb-from-prop}
    \acr_k(F)\ \ge\ 1-\frac{\gamma(1-\gamma)}{2}\frac{\log k}{k}-\frac{M}{k}.
    \end{equation}
    Since $\gamma(1-\gamma)\le 1/4$ for all $\gamma\in(0,1)$, \eqref{eq:acr-lb-from-prop} implies
    \begin{equation}\label{eq:acr-lb-unif}
    \acr_k(F)\ \ge\ 1-\frac{1}{8}\frac{\log k}{k}-\frac{M}{k}.
    \end{equation}
    Choose $k_\varepsilon$ sufficiently large so that
    \[
    \frac{M}{k}\le \frac{\varepsilon}{8}\frac{\log k}{k}
    \qquad\text{for all }k\ge k_\varepsilon,
    \]
    specifically one may take
    $k_\varepsilon:=\left\lceil \exp(8M/\varepsilon)\right\rceil$). Then for all $k\ge k_\varepsilon$,
    \[
    \frac{1}{8}\frac{\log k}{k}+\frac{M}{k}
    \le \frac{1+\varepsilon}{8}\frac{\log k}{k}.
    \]
    Combining with \eqref{eq:acr-lb-unif} yields
    \[
    \acr_k(F)\ \ge\ 1-\frac{1+\varepsilon}{8}\frac{\log k}{k},
    \qquad \forall k\ge k_\varepsilon.
    \]
    This proves the corollary, and the bound is independent of $\gamma$ because we only used the universal inequality
    $\gamma(1-\gamma)\le 1/4$.

\end{proof}

\section{Proofs of results in Section \ref{sec:ce}} 

\subsection{Proof of Theorem~\ref{thm:acr-ce}}

\begin{proof}[Proof of Theorem~\ref{thm:acr-ce}]
    Recall that the asymptotic competitive ratio of the CE heuristic is
    \[
    \apx_k(F):=\lim_{n\to\infty}\frac{\vce(n,k)}{\mu_{n,k}},
    \]
    whenever the limit exists.
    
    We first treat the Fr\'echet case $\gamma\in(0,1)$, and then discuss $\gamma\notin(0,1)$.
    
    \subsubsection*{Step 1: $\gamma\in(0,1)$}
    Let $\gamma\in(0,1)$ be fixed. By Proposition~\ref{prop:ce-value-gamma-ge-0}, for each fixed $k\ge1$,
    \begin{equation}\label{eq:ce-asymp-n}
    \vce(n,k)=w_k\,F^{\leftarrow}\!\Bigl(1-\frac{1}{n}\Bigr)\,[1+o(1)],
    \qquad n\to\infty,
    \end{equation}
    where $\{w_k\}$ satisfies $w_1=(1-\gamma^2)^{-1}$ and
    \begin{equation}\label{eq:wk-rec}
    w_k=\frac{k}{k+\gamma}w_{k-1}+\frac{1}{k+\gamma}\frac{k^{1-\gamma}}{1-\gamma},
    \qquad k\ge2.
    \end{equation}
    On the other hand, Theorem~\ref{thm:prophet} gives the oracle asymptotics
    \begin{equation}\label{eq:oracle-asymp-n}
    \mu_{n,k}=
    \frac{\Gamma(k+1-\gamma)}{(1-\gamma)\Gamma(k)}\,
    F^{\leftarrow}\!\Bigl(1-\frac{1}{n}\Bigr)\,[1+o(1)],
    \qquad n\to\infty.
    \end{equation}
    Combining \eqref{eq:ce-asymp-n} and \eqref{eq:oracle-asymp-n} yields
    \begin{equation}\label{eq:apx-in-wk}
    \apx_k(F)
    =\lim_{n\to\infty}\frac{\vce(n,k)}{\mu_{n,k}}
    =
    (1-\gamma)\frac{\Gamma(k)}{\Gamma(k+1-\gamma)}\,w_k.
    \end{equation}
    Therefore, it remains to express $w_k$ in closed form and substitute into \eqref{eq:apx-in-wk}.
    
    \begin{lemma}\label{lem:wk-closed-form}
    Let $\gamma\in(0,1)$ and define $\{w_k\}_{k\ge1}$ by $w_1=(1-\gamma^2)^{-1}$ and \eqref{eq:wk-rec}.
    Then for every $k\ge2$,
    \begin{equation}\label{eq:wk-closed}
    w_k
    =
    \frac{\Gamma(k+1)\Gamma(2+\gamma)}{\Gamma(k+1+\gamma)}\,w_1
    \;+\;
    \frac{1}{1-\gamma}\,
    \frac{\Gamma(k+1)}{\Gamma(k+1+\gamma)}
    \sum_{r=2}^{k}\frac{\Gamma(r+\gamma)}{\Gamma(r+1)}\,r^{1-\gamma}.
    \end{equation}
    In particular, using $\Gamma(1+\gamma)=\gamma\Gamma(\gamma)$ and $w_1=(1-\gamma^2)^{-1}$,
    the first term in \eqref{eq:wk-closed} equals
    \[
    \frac{1}{1-\gamma}\,\frac{\Gamma(k+1)}{\Gamma(k+1+\gamma)} \Gamma(1+\gamma).
    \]
    Hence, for every $k\ge1$,
    \begin{equation}\label{eq:wk-sum}
    w_k
    =
    \frac{1}{1-\gamma}\,
    \frac{\Gamma(k+1)}{\Gamma(k+1+\gamma)}
    \sum_{r=1}^{k}\frac{\Gamma(r+\gamma)}{\Gamma(r+1)}\,r^{1-\gamma}.
    \end{equation}
    \end{lemma}
    
    \begin{proof}
    Multiply \eqref{eq:wk-rec} by $\Gamma(k+1+\gamma)/\Gamma(k+1)$. Using
    \[
    \frac{\Gamma(k+1+\gamma)}{\Gamma(k+1)}\cdot\frac{k}{k+\gamma}
    =
    \frac{\Gamma(k+\gamma)}{\Gamma(k)}
    \qquad\text{and}\qquad
    \frac{\Gamma(k+1+\gamma)}{\Gamma(k+1)}\cdot\frac{1}{k+\gamma}
    =
    \frac{\Gamma(k+\gamma)}{\Gamma(k+1)},
    \]
    we obtain the telescoping form
    \[
    \frac{\Gamma(k+1+\gamma)}{\Gamma(k+1)}w_k
    =
    \frac{\Gamma(k+\gamma)}{\Gamma(k)}w_{k-1}
    +
    \frac{1}{1-\gamma}\,\frac{\Gamma(k+\gamma)}{\Gamma(k+1)}\,k^{1-\gamma}.
    \]
    Define
    \[
    y_k:=\frac{\Gamma(k+1+\gamma)}{\Gamma(k+1)}w_k.
    \]
    Then for $k\ge2$,
    \begin{equation}\label{eq:yk-rec}
    y_k=y_{k-1}+\frac{1}{1-\gamma}\,\frac{\Gamma(k+\gamma)}{\Gamma(k+1)}\,k^{1-\gamma}.
    \end{equation}
    Summing \eqref{eq:yk-rec} from $2$ to $k$ yields
    \[
    y_k
    =
    y_1+\frac{1}{1-\gamma}\sum_{r=2}^{k}\frac{\Gamma(r+\gamma)}{\Gamma(r+1)}\,r^{1-\gamma}.
    \]
    Since $y_1=\Gamma(2+\gamma)w_1=\Gamma(1+\gamma)(1+\gamma)w_1$, substituting back
    $w_k=\frac{\Gamma(k+1)}{\Gamma(k+1+\gamma)}y_k$ gives \eqref{eq:wk-closed}.
    
    Finally, observe that $w_1=(1-\gamma^2)^{-1}=\bigl((1-\gamma)(1+\gamma)\bigr)^{-1}$ and
    $\Gamma(2+\gamma)=(1+\gamma)\Gamma(1+\gamma)$, so
    \[
    \frac{\Gamma(k+1)\Gamma(2+\gamma)}{\Gamma(k+1+\gamma)}w_1
    =
    \frac{1}{1-\gamma}\frac{\Gamma(k+1)}{\Gamma(k+1+\gamma)}
    \cdot
    \Gamma(1+\gamma).
    \]
    This is exactly the $r=1$ term in the sum \eqref{eq:wk-sum}, since
    \[
    \frac{\Gamma(1+\gamma)}{\Gamma(2)}\,1^{1-\gamma}=\Gamma(1+\gamma).
    \]
    Therefore \eqref{eq:wk-closed} can be rewritten as \eqref{eq:wk-sum}, completing the proof.
    \end{proof}
    
    Substituting the closed form \eqref{eq:wk-sum} into \eqref{eq:apx-in-wk} gives
    \[
    \apx_k(F)
    =
    (1-\gamma)\frac{\Gamma(k)}{\Gamma(k+1-\gamma)}
    \cdot
    \frac{1}{1-\gamma}\,
    \frac{\Gamma(k+1)}{\Gamma(k+1+\gamma)}
    \sum_{r=1}^{k}\frac{\Gamma(r+\gamma)}{\Gamma(r+1)}\,r^{1-\gamma},
    \]
    which simplifies to the stated formula
    \[
    \apx_k(F)
    =
    \frac{\Gamma(k)\Gamma(k+1)}{\Gamma(k+1-\gamma)\Gamma(k+1+\gamma)}
    \sum_{r=1}^{k}\frac{\Gamma(r+\gamma)}{\Gamma(r+1)}\,r^{1-\gamma}.
    \]
    
    \subsubsection*{Step 2: $\gamma\notin(0,1)$}
    When $\gamma\notin(0,1)$, the asymptotic competitive ratio equals $1$.

    In fact, if $\gamma=0$, Theorem~\ref{thm:prophet} yields
    \[
    \mu_{n,k}
    =
    k\,F^{\leftarrow}\!\Bigl(1-\frac{1}{n}\Bigr)
    +
    a_n\sum_{r=1}^k\bigl(-\psi(r)\bigr),
    \qquad n\to\infty,
    \]
    where $a_n=o\!\bigl(F^{\leftarrow}(1-\frac{1}{n})\bigr)$. Hence
    \[
    \mu_{n,k}
    =
    k\,F^{\leftarrow}\!\Bigl(1-\frac{1}{n}\Bigr)\,[1+o(1)].
    \]
    On the other hand, the $\gamma=0$ version of the CE value asymptotics from
    Proposition~\ref{prop:ce-value-gamma-eq-0} gives
    \[
    \vce(n,k)
    =
    k\,F^{\leftarrow}\!\Bigl(1-\frac{1}{n}\Bigr)\,[1+o(1)].
    \]
    Therefore $\apx_k(F)=\lim_{n\to\infty}\vce(n,k)/\mu_{n,k}=1$.
    
    If $\gamma<0$, let $x^\ast<\infty$ be the right endpoint of $F$. Theorem~\ref{thm:prophet} gives
    \[
    \mu_{n,k}
    =
    k\,x^\ast
    -
    \frac{\Gamma(k+1-\gamma)}{(1-\gamma)\Gamma(k)}
    \Bigl(x^\ast-F^{\leftarrow}\!\Bigl(1-\frac{1}{n}\Bigr)\Bigr)
    \,[1+o(1)].
    \] 
    Proposition~\ref{prop:ce-value-gamma-le-0} yields the same leading term $kx^\ast$ for $\vce(n,k)$, and $F^{\leftarrow}(1-\frac{1}{n}) \to x^\ast$ as $n\to \infty$. Dividing the two expansions implies
    \[
    \frac{\vce(n,k)}{\mu_{n,k}}=1+o(1),
    \qquad n\to\infty,
    \]
    and hence $\apx_k(F)=1$.
    
    This establishes the second line of Theorem~\ref{thm:acr-ce}.
    
\end{proof}

\subsection{Proof of Proposition~\ref{prop:ce-large-k}}

\begin{proof}
Fix $\gamma\in(0,1)$ and write the expression in Theorem~\ref{thm:acr-ce} as
\begin{equation}\label{eq:ce-factorization}
\apx_k(F)=P_k(\gamma)\,S_k(\gamma),
\qquad
P_k(\gamma):=\frac{\Gamma(k)\Gamma(k+1)}{\Gamma(k+1-\gamma)\Gamma(k+1+\gamma)},
\qquad
S_k(\gamma):=\sum_{r=1}^k a_r(\gamma),
\end{equation}
where
\[
a_r(\gamma):=\frac{\Gamma(r+\gamma)}{\Gamma(r+1)}\,r^{1-\gamma}.
\]
Let $c:=\gamma(1-\gamma)\in(0,1/4]$.

\begin{lemma}\label{lem:ce-prefactor}
For $\gamma\in(0,1)$,
\[
P_k(\gamma)=\frac{1}{k}\Bigl(1+O\!\left(\frac{1}{k}\right)\Bigr),
\qquad k\to\infty.
\]
More precisely,
\[
P_k(\gamma)=\frac{1}{k}\left(1-\frac{\gamma^2}{k}+O\!\left(\frac{1}{k^2}\right)\right).
\]
\end{lemma}

\begin{proof}
Use Lemma~\ref{lem:Gamma-ratio-exp}, 
\[
\frac{\Gamma(k)}{\Gamma(k+1-\gamma)}
=
k^{\gamma-1}\left(1+\frac{\gamma(1-\gamma)}{2k}+O\!\left(\frac{1}{k^2}\right)\right),
\]
and
\begin{align*}
\frac{\Gamma(k+1)}{\Gamma(k+1+\gamma)}
&=k^{-\gamma} \frac{k^{\gamma+1} \Gamma(k)}{\Gamma(k+\gamma+1)}\\
&= k^{-\gamma}\left(1-\frac{\gamma(1+\gamma)}{2k}+O\!\left(\frac{1}{k^2}\right)\right).\\
\end{align*}
Multiplying yields
\[
P_k(\gamma)
=\frac{\Gamma(k)}{\Gamma(k+1-\gamma)}\cdot \frac{\Gamma(k+1)}{\Gamma(k+1+\gamma)}
=
\frac{1}{k}\left(1-\frac{\gamma^2}{k}+O\!\left(\frac{1}{k^2}\right)\right),
\]
as claimed.
\end{proof}

\begin{lemma}\label{lem:ce-summand-expansion}
For each fixed $\gamma\in(0,1)$,
\[
a_r(\gamma)
=
1-\frac{c}{2r}+O\!\left(\frac{1}{r^2}\right),
\qquad r\to\infty,
\]
where $c=\gamma(1-\gamma)$.
\end{lemma}

\begin{proof}
An extension of Lemma~\ref{lem:Gamma-ratio-exp} would be: for fixed $a,b$,
\[
\frac{\Gamma(k+a)}{\Gamma(k+b)}
=
k^{a-b}\left(1+\frac{(a-b)(a+b-1)}{2k}+O\!\left(\frac{1}{k^2}\right)\right),
\qquad k\to\infty.
\]
Apply the formula with $a=\gamma$ and $b=1$:
\[
\frac{\Gamma(r+\gamma)}{\Gamma(r+1)}
=
r^{\gamma-1}\left(1+\frac{(\gamma-1)\gamma}{2r}+O\!\left(\frac{1}{r^2}\right)\right)
=
r^{\gamma-1}\left(1-\frac{c}{2r}+O\!\left(\frac{1}{r^2}\right)\right),
\]
and then multiply by $r^{1-\gamma}$ to obtain
\[
a_r(\gamma)=\frac{\Gamma(r+\gamma)}{\Gamma(r+1)}\,r^{1-\gamma}
=
1-\frac{c}{2r}+O\!\left(\frac{1}{r^2}\right).
\]
\end{proof}

\begin{lemma}\label{lem:ce-sum-expansion}
As $k\to\infty$,
\[
S_k(\gamma)=k-\frac{c}{2}\log k+O(1),
\qquad c=\gamma(1-\gamma).
\]
\end{lemma}

\begin{proof}
By Lemma~\ref{lem:ce-summand-expansion}, there exists $r_0$ such that for all $r\ge r_0$,
\[
a_r(\gamma)=1-\frac{c}{2r}+\epsilon_r,
\qquad\text{with}\qquad
\epsilon_r=O\!\left(\frac{1}{r^2}\right).
\]
Summing from $r=1$ to $k$ gives
\[
S_k(\gamma)
=
k-\frac{c}{2}\sum_{r=1}^k\frac{1}{r}
+\sum_{r=r_0}^k\epsilon_r
+O(1).
\]
Since $\sum_{r=r_0}^\infty \epsilon_r$ converges (because $\epsilon_r=O(r^{-2})$), we have
$\sum_{r=r_0}^k\epsilon_r=O(1)$. Moreover, the harmonic numbers satisfy
$\sum_{r=1}^k \frac{1}{r}=\log k+O(1)$. Therefore,
\[
S_k(\gamma)=k-\frac{c}{2}\log k+O(1),
\]
as claimed.
\end{proof}

Combining Lemma~\ref{lem:ce-prefactor} and Lemma~\ref{lem:ce-sum-expansion} in \eqref{eq:ce-factorization} yields
\[
\apx_k(F)
=
\left(\frac{1}{k}\Bigl(1+O\!\left(\frac{1}{k}\right)\Bigr)\right)
\left(k-\frac{c}{2}\log k+O(1)\right)
=
1-\frac{c}{2}\frac{\log k}{k}+O\!\left(\frac{1}{k}\right),
\qquad k\to\infty,
\]
with $c=\gamma(1-\gamma)$. Hence there exist $k_\ast'$ and $M_1'<\infty$ such that for all $k\ge k_\ast'$,
\[
\left|\apx_k(F)-\left(1-\frac{\gamma(1-\gamma)}{2}\frac{\log k}{k}\right)\right|\le \frac{M_1'}{k}.
\]
For the finitely many indices $1\le k<k_\ast'$, define
\[
M_2':=\max_{1\le k<k_\ast'}
k\left|\apx_k(F)-\left(1-\frac{\gamma(1-\gamma)}{2}\frac{\log k}{k}\right)\right|.
\]
Then $M_2'<\infty$. Taking $M':=\max\{M_1',M_2'\}$ gives \eqref{eq:dp-large-k} for every $k\ge1$.
\end{proof}

\subsection{Proof of Corollary~\ref{cor:ce-large-k}}

\begin{proof}
Fix $\varepsilon>0$. If $\gamma\notin(0,1)$, then by Theorem~\ref{thm:acr-ce} we have $\apx_k(F)=1$ for all $k$,
so the inequality holds trivially for every $k$.

Now assume $\gamma\in(0,1)$. By Proposition~\ref{prop:ce-large-k}, there exists $M'<\infty$ such that for all
$k\ge1$,
\begin{equation}\label{eq:ce-lb-from-prop}
\apx_k(F)\ \ge\ 1-\frac{\gamma(1-\gamma)}{2}\frac{\log k}{k}-\frac{M'}{k}.
\end{equation}
Since $\gamma(1-\gamma)\le \frac14$ for all $\gamma\in(0,1)$, \eqref{eq:ce-lb-from-prop} implies the uniform bound
\begin{equation}\label{eq:ce-lb-unif}
\apx_k(F)\ \ge\ 1-\frac{1}{8}\frac{\log k}{k}-\frac{M'}{k}.
\end{equation}
Choose $k'_\varepsilon$ sufficiently large such that
\[
\frac{M'}{k}\le \frac{\varepsilon}{8}\frac{\log k}{k}
\qquad\text{for all }k\ge k'_\varepsilon,
\]
specifically one may take $k'_\varepsilon:=\left\lceil \exp(8M'/\varepsilon)\right\rceil$. Then for all $k\ge k'_\varepsilon$,
\[
\frac{1}{8}\frac{\log k}{k}+\frac{M'}{k}
\le \frac{1+\varepsilon}{8}\frac{\log k}{k}.
\]
Combining with \eqref{eq:ce-lb-unif} yields
\[
\apx_k(F)\ \ge\ 1-\frac{1+\varepsilon}{8}\frac{\log k}{k},
\qquad \forall k\ge k'_\varepsilon,
\]
which is the desired claim. The bound is independent of $\gamma$ because it uses only the universal inequality
$\gamma(1-\gamma)\le 1/4$.
\end{proof}

\subsection{Proof of Theorem~\ref{thm:regret-ce-dp}}

\begin{lemma}\label{lem:dp-sk-constant}
Let $\gamma\in(0,1)$ and let $\{v_k\}_{k\ge1}$ be the sequence in Theorem~\ref{thm:acr-dp}. Define
\[
s_k:=(1-\gamma)\,v_k^{1/(1-\gamma)},\qquad k\ge1.
\]
There exists a constant $d^{\mathrm{DP}}_\gamma\in\R$ such that
\[
s_k = k-\frac{\gamma}{2}\log k + d^{\mathrm{DP}}_\gamma + o(1),
\qquad k\to\infty.
\]
\end{lemma}

\begin{proof}
From Lemma~\ref{lem:sk-asymp}, we have
\begin{equation}\label{eq:sk-first-order}
\frac{s_k}{k}
=
1-\frac{\gamma}{2}\frac{\log k}{k}+O\!\left(\frac{1}{k}\right),
\qquad\text{equivalently}\qquad
s_k=k-\frac{\gamma}{2}\log k+O(1),
\end{equation}
and
\begin{equation}\label{eq:sk-increment}
s_k-s_{k-1}
=
1-\frac{\gamma}{2s_k}+O\!\left(\frac{1}{s_k^2}\right).
\end{equation}
It remains to upgrade the $O(1)$ remainder in \eqref{eq:sk-first-order} to convergence.

Define
\[
t_k:=s_k-k+\frac{\gamma}{2}\log k.
\]
Using \eqref{eq:sk-increment} and $\log k-\log(k-1)=\frac{1}{k}+O(\frac{1}{k^2})$ yields
\begin{align*}
t_k-t_{k-1}
&=\bigl((s_k-s_{k-1})-1\bigr)+\frac{\gamma}{2}\bigl(\log k-\log(k-1)\bigr) \\
&=-\frac{\gamma}{2s_k}+\frac{\gamma}{2k}+O\!\left(\frac{1}{s_k^2}\right)+O\!\left(\frac{1}{k^2}\right).
\end{align*}
From \eqref{eq:sk-first-order}, $s_k=k+O(\log k)$, so
\[
\frac{1}{s_k}=\frac{1}{k}+O\!\left(\frac{\log k}{k^2}\right),
\qquad
\frac{1}{s_k^2}=O\!\left(\frac{1}{k^2}\right).
\]
Therefore
\[
t_k-t_{k-1}=O\!\left(\frac{\log k}{k^2}\right),
\]
and since $\sum_{k\ge2} (\log k)/k^2<\infty$, the sequence $\{t_k\}$ is convergent.
Let $d^{\mathrm{DP}}_\gamma:=\lim_{k\to\infty} t_k$. Then
\[
s_k = k-\frac{\gamma}{2}\log k + d^{\mathrm{DP}}_\gamma + o(1),
\]
which proves the lemma.
\end{proof}

\begin{lemma}\label{lem:ce-sum-constant}
Fix $\gamma\in(0,1)$ and set $c:=\gamma(1-\gamma)$. Define the following as the same in ~\eqref{eq:ce-factorization},
\[
S_k(\gamma):=\sum_{r=1}^k a_r(\gamma),
\]
where
\[
a_r(\gamma):=\frac{\Gamma(r+\gamma)}{\Gamma(r+1)}\,r^{1-\gamma}
\]
There exists a constant $d^{\mathrm{CE}}_\gamma\in\R$ such that
\[
S_k(\gamma)=k-\frac{c}{2}\log k + d^{\mathrm{CE}}_\gamma + o(1),
\qquad k\to\infty.
\]
% Equivalently, the sequence
% \[
% y_k:=\frac{1}{1-\gamma}S_k(\gamma)
% \]
% satisfies
% \[
% y_k=\frac{k}{1-\gamma}-\frac{\gamma}{2}\log k+\frac{d^{\mathrm{CE}}_\gamma}{1-\gamma}+o(1).
% \]
\end{lemma}

\begin{proof}
From Lemma ~\ref{lem:ce-sum-expansion}, 
% Let $P_k(\gamma):=\dfrac{\Gamma(k)\Gamma(k+1)}{\Gamma(k+1-\gamma)\Gamma(k+1+\gamma)}$, so that
% Theorem~\ref{thm:acr-ce} gives
% \[
% \apx_k(F)=P_k(\gamma)\,S_k(\gamma).
% \]
% By Proposition~\ref{prop:ce-large-k}, for some $M'<\infty$ and all $k\ge1$,
% \[
% \apx_k(F)=1-c\frac{\log k}{k}+\eta_k,
% \qquad\text{with}\qquad |\eta_k|\le \frac{M'}{k}.
% \]
% A gamma-ratio expansion (via Stirling) yields
% \[
% k\,P_k(\gamma)=1+O\!\left(\frac{1}{k}\right),
% \qquad\text{hence}\qquad
% \frac{1}{P_k(\gamma)}=k\Bigl(1+O\!\left(\frac{1}{k}\right)\Bigr).
% \]
Therefore,
\begin{equation}\label{eq:Sk-O1}
S_k(\gamma)
= k-\frac{c}{2}\log k + O(1).
\end{equation}

It remains to upgrade \eqref{eq:Sk-O1} to convergence.

Using again the gamma-ratio expansion,
\[
\frac{\Gamma(r+\gamma)}{\Gamma(r+1)}
=
r^{\gamma-1}\left(1-\frac{\gamma(1-\gamma)}{2r}+O\!\left(\frac{1}{r^2}\right)\right),
\qquad r\to\infty,
\]
so
\[
a_r(\gamma)=1-\frac{c}{2r}+O\!\left(\frac{1}{r^2}\right),
\qquad r\to\infty.
\]
Define the remainder
\[
b_r:=a_r(\gamma)-1+\frac{c}{r}.
\]
Then $b_r=O(r^{-2})$, hence $\sum_{r=1}^\infty b_r$ converges. 
\[
S_k(\gamma)
=\sum_{r=1}^k\left(1-\frac{c}{2r}+b_r\right)
= k-\frac{c}{2}H_k+\sum_{r=1}^k b_r.
\]
Since $H_k=\log k+\gamma_{\mathrm{E}}+o(1)$ when $k\to \infty$, and $\sum_{r=1}^k b_r$ converges to a finite limit,
it follows that $S_k(\gamma)-k+c\log k$ converges to some finite constant $d^{\mathrm{CE}}_\gamma$.
\end{proof}

\begin{proof}[Proof of Theorem~\ref{thm:regret-ce-dp}]

Assume first that $\gamma\in(0,1)$. For each fixed $k$, the DP and CE value asymptotics give
\begin{equation}\label{eq:dp-ce-fixed-k}
\vdp(n,k)=\alpha_k\,U(n)\,[1+o(1)],
\qquad
\vce(n,k)=\beta_k\,U(n)\,[1+o(1)],
\qquad n\to\infty,
\end{equation}
where
\[
\alpha_k=\frac{v_k}{(1-\gamma)^\gamma},
\qquad
\beta_k=w_k,
\]
with $\{v_k\}$ as in Theorem~\ref{thm:acr-dp} and $\{w_k\}$ the CE coefficient sequence in ~\ref{eq:ck-recursion}.
Subtracting in \eqref{eq:dp-ce-fixed-k} yields, for each fixed $k$,
\[
\lim_{n\to\infty}\frac{\vdp(n,k)-\vce(n,k)}{k^{-\gamma}U(n)}
=
k^\gamma(\alpha_k-\beta_k).
\]
Therefore, the theorem reduces to showing that $\lim_{k\to\infty}k^\gamma(\alpha_k-\beta_k)$ exists and is finite.

Recall $s_k:=(1-\gamma)\,v_k^{1/(1-\gamma)}$, so that
\[
\alpha_k=\frac{v_k}{(1-\gamma)^\gamma}=\frac{s_k^{1-\gamma}}{1-\gamma}.
\]
By Lemma~\ref{lem:dp-sk-constant}, there exists $d^{\mathrm{DP}}_\gamma\in\R$ such that
\[
s_k = k-\frac{\gamma}{2}\log k + d^{\mathrm{DP}}_\gamma + o(1).
\]
Let $h_k:=-(\gamma/2)\log k+d^{\mathrm{DP}}_\gamma+o(1)$, so that $s_k=k+h_k$ with $h_k=O(\log k)$.
Then
\[
s_k^{1-\gamma}
=
k^{1-\gamma}\Bigl(1+\frac{h_k}{k}\Bigr)^{1-\gamma}
=
k^{1-\gamma}+(1-\gamma)k^{-\gamma}h_k+o(k^{-\gamma}),
\]
because $\frac{h_k}{k}\to0$ and $k^{1-\gamma}\bigl(\frac{h_k}{k}\bigr)^2=o(k^{-\gamma})$.
Dividing by $(1-\gamma)$ gives
\begin{equation}\label{eq:alpha-exp}
\alpha_k
=
\frac{k^{1-\gamma}}{1-\gamma}
-\frac{\gamma}{2}\,k^{-\gamma}\log k
+d^{\mathrm{DP}}_\gamma\,k^{-\gamma}
+o\!\bigl(k^{-\gamma}\bigr).
\end{equation}
Recall the definition in Lemma~\ref{lem:wk-closed-form},
\[
y_k:=\frac{\Gamma(k+1+\gamma)}{\Gamma(k+1)}\,w_k,
\]
and
\begin{equation}\label{eq:y-sum}
y_k
=
\frac{1}{1-\gamma}\sum_{r=1}^k a_r(\gamma),
\qquad
a_r(\gamma):=\frac{\Gamma(r+\gamma)}{\Gamma(r+1)}\,r^{1-\gamma}.
\end{equation}
Lemma~\ref{lem:ce-sum-constant} shows that there exists $d^{\mathrm{CE}}_\gamma\in\R$ such that
\[
\sum_{r=1}^k a_r(\gamma)=k-\frac{\gamma(1-\gamma)}{2}\log k+d^{\mathrm{CE}}_\gamma+o(1).
\]
Combining this with \eqref{eq:y-sum}
yields the existence of $\tilde d^{\mathrm{CE}}_\gamma\in\R$ such that
\begin{equation}\label{eq:y-asymp}
y_k
=
\frac{k}{1-\gamma}-\frac{\gamma}{2}\log k+\tilde d^{\mathrm{CE}}_\gamma+o(1).
\end{equation}
Next, use the gamma-ratio expansion
\begin{equation}\label{eq:Gamma-ratio-ce}
\frac{\Gamma(k+1)}{\Gamma(k+1+\gamma)}
=
k^{-\gamma}\Bigl(1+O\!\left(\frac{1}{k}\right)\Bigr),
\qquad k\to\infty.
\end{equation}
Since $w_k=y_k\,\Gamma(k+1)/\Gamma(k+1+\gamma)$, \eqref{eq:y-asymp}--\eqref{eq:Gamma-ratio-ce} give
\begin{equation}\label{eq:beta-exp}
\beta_k=w_k
=
\frac{k^{1-\gamma}}{1-\gamma}
-\frac{\gamma}{2}\,k^{-\gamma}\log k
+\tilde d^{\mathrm{CE}}_\gamma\,k^{-\gamma}
+o\!\bigl(k^{-\gamma}\bigr).
\end{equation}
Subtracting \eqref{eq:beta-exp} from \eqref{eq:alpha-exp} yields
\[
\alpha_k-\beta_k
=
\bigl(d^{\mathrm{DP}}_\gamma-\tilde d^{\mathrm{CE}}_\gamma\bigr)\,k^{-\gamma}
+o\!\bigl(k^{-\gamma}\bigr),
\qquad k\to\infty.
\]
Hence the limit exists and
\[
\lim_{k\to\infty}k^\gamma(\alpha_k-\beta_k)
=
d^{\mathrm{DP}}_\gamma-\tilde d^{\mathrm{CE}}_\gamma
=:c_\gamma.
\]
Combining with the reduction for $n$ gives
\[
\lim_{k\to\infty}\ \lim_{n\to\infty}\ 
\frac{\vdp(n,k)-\vce(n,k)}{k^{-\gamma}U(n)}
=
c_\gamma,
\qquad \gamma\in(0,1).
\]
Since by definition $\vdp(n,k)\ge \vce(n,k)$ for all $(n,k)$, we have $\alpha_k\ge \beta_k$ for all $k$ and thus $c_\gamma\ge0$, and by numerical validation, we infer $c_\gamma > 0;$ see Figure~\ref{fig:third_order}.

\begin{figure}[htbp]
  \centering
  \begin{subfigure}[b]{0.48\textwidth}
    \centering
    \includegraphics[height=5cm, keepaspectratio]{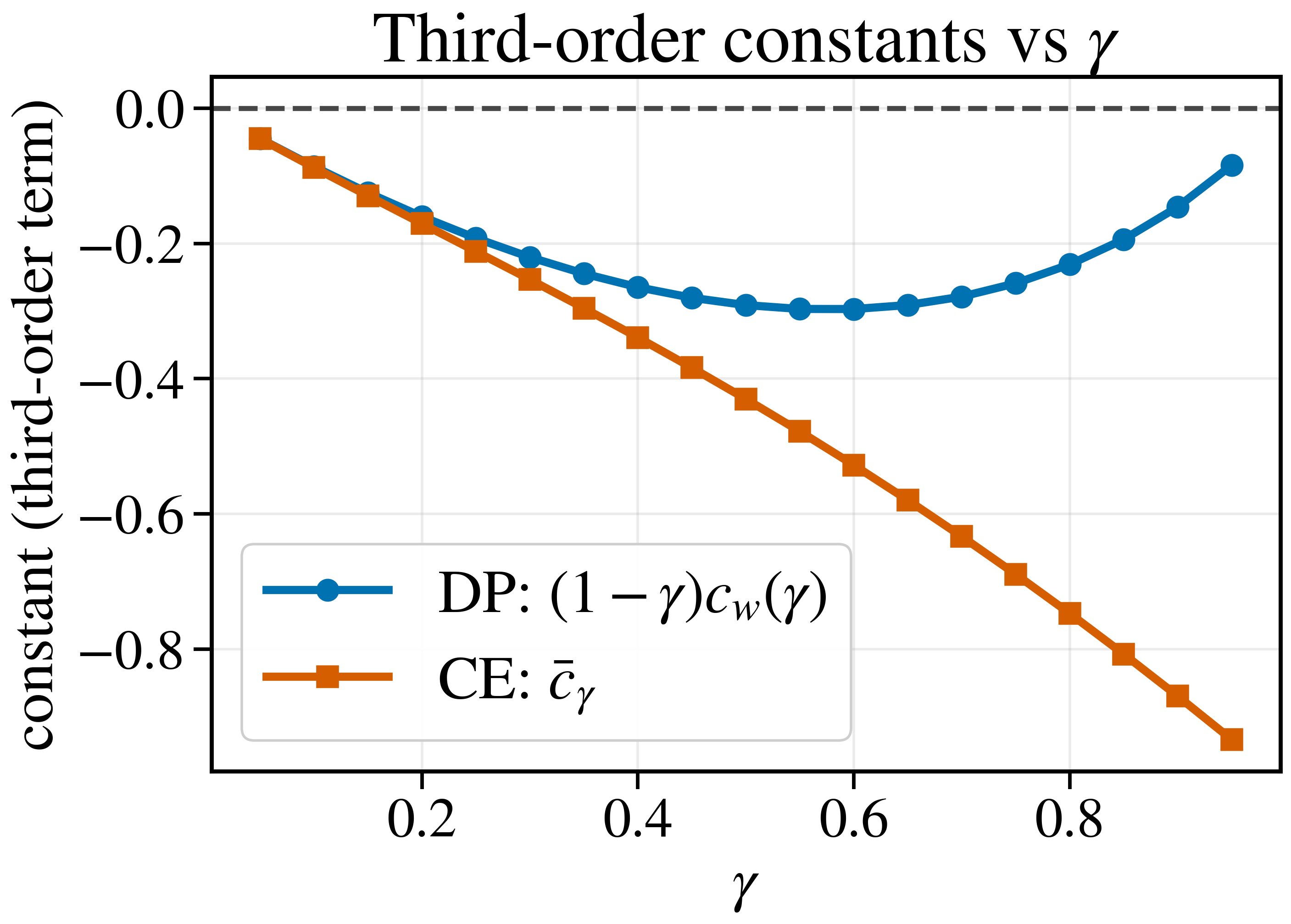}
    \caption{$(1-\gamma)c_\omega(\gamma)$ and $\bar c_\gamma$}
  \end{subfigure}
  \hfill
  \begin{subfigure}[b]{0.48\textwidth}
    \centering
    \includegraphics[height=5cm, keepaspectratio]{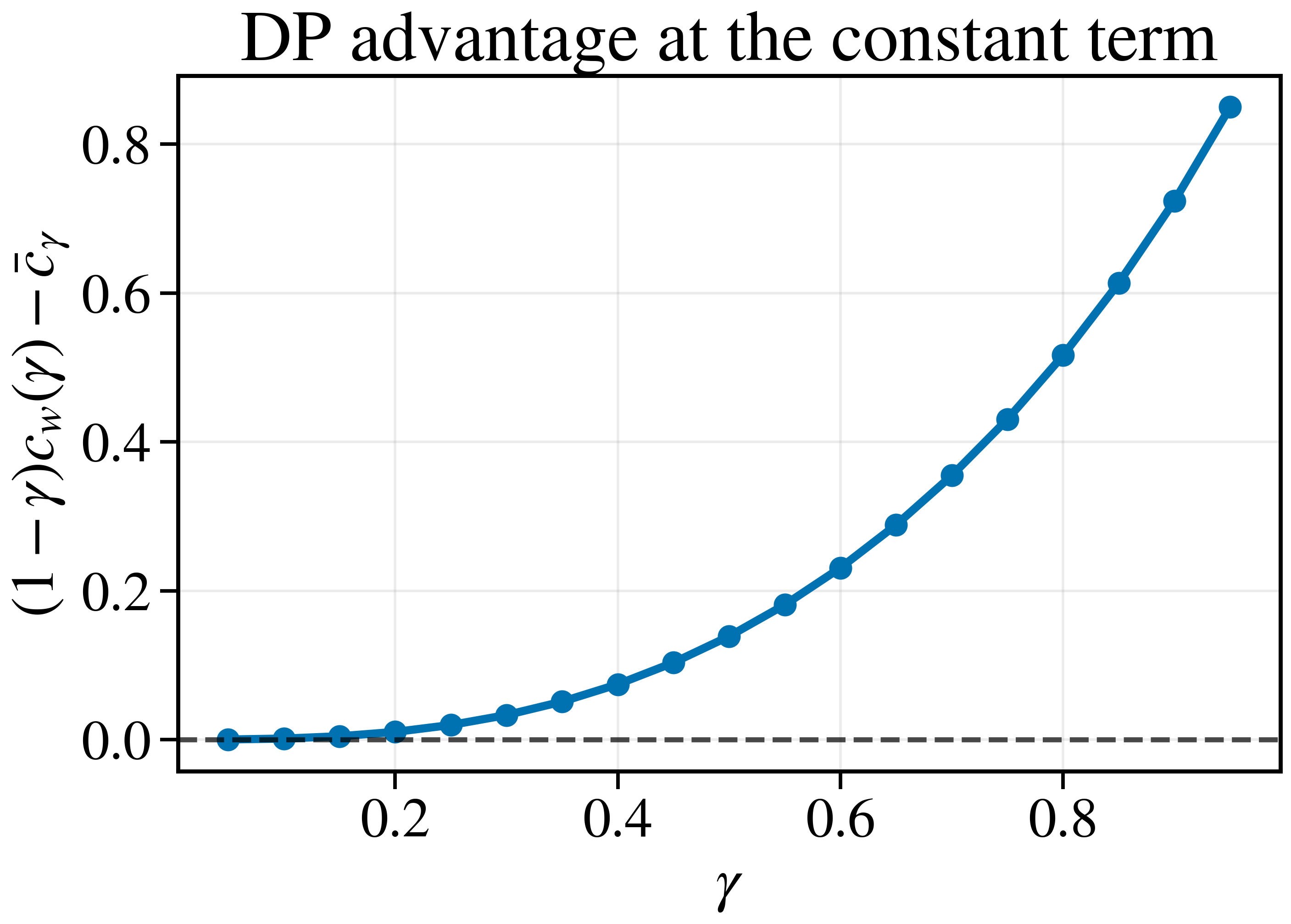}
    \caption{$(1-\gamma)c_\omega(\gamma) -\bar c_\gamma$}
  \end{subfigure}
  \caption{Finer Comparison between DP and CE}
  \label{fig:third_order}
\end{figure}

When $\gamma=0$, the DP and CE fixed-$k$ asymptotics are governed by the same limiting recursion (and the same
initial condition), hence $\alpha_k=\beta_k$ for every fixed $k$ and therefore $c_0=0$.

If $\gamma\in(0,1)$, then $U(n)\to\infty$ as $n\to\infty$. Moreover, the previous steps imply that for large $k$,
\[
\lim_{n\to\infty}\frac{\vdp(n,k)-\vce(n,k)}{U(n)}
=
\alpha_k-\beta_k
=\Theta(k^{-\gamma}),
\]
so as $n\to \infty$,
\[
\vdp(n,k) - \vce(n,k) =\Theta\left( \left(\frac{n}{k}\right)^{\gamma}\right),
\]
and therefore for any fixed sufficiently large $k$ with $\alpha_k>\beta_k$, the regret $\vdp(n,k)-\vce(n,k)$ diverges as
$n\to\infty$.

If $\gamma <0$, they will converge to the same leading term $kx^\ast$, and
\[
\vdp(n,k) - \vce(n,k)  \longrightarrow0.
\]
    
\end{proof}

\subsection{Proof of Proposition~\ref{prop:regret-ce-dp-pareto}}

\begin{proof}
    Fix $\gamma\in(0,1)$ and let $F$ be the Pareto distribution on $[1,\infty)$ with tail
    \[
    \bar F(x)=x^{-1/\gamma},\qquad x\ge1.
    \]
    Then $F$ satisfies the extreme value condition with index $\gamma$ and has finite mean since $1/\gamma>1$.
    Its right end-point is $x^\ast=\infty$, and the high quantile function is explicit:
    \[
    U(n):=F^{\leftarrow}\!\Bigl(1-\frac{1}{n}\Bigr)=n^\gamma,\qquad n\ge1.
    \]
    Moreover, for every $t\ge1$ we have the exact tail-integral identity
    \begin{equation}\label{eq:Pareto-tail-int}
    \int_t^\infty \bar F(u)\,du
    =\int_t^\infty u^{-1/\gamma}\,du
    =\frac{\gamma}{1-\gamma}\,t^{1-1/\gamma}
    =\frac{\gamma}{1-\gamma}\,t\,\bar F(t),
    \end{equation}
    which eliminates any approximation error in the recursions.
    
    Consider now the joint regime $n\to\infty$ with $k=k(n)=o(n)$. For this Pareto choice of $F$, the DP recursion
    (Lemma~\ref{lem:dp-recursive}) and the CE recursion (Lemma~\ref{lem:ce-recursive}) simplify using \eqref{eq:Pareto-tail-int} in a way that is uniform
    whenever $k/n\to0$. Also, for the pareto distribution, whenever $k=o(n)$, the followings hold exactly in Proposition~\ref{prop:dp-value-gamma-ge-0} and ~\ref{prop:ce-value-gamma-ge-0},
    \[
    \frac{U(n/k)}{U(n)} = k^{-\gamma}, \quad \frac{nR(U(n)y)}{U(n)}
    =
    \frac{\gamma}{1-\gamma}y^{-(1-\gamma)/\gamma}
    \]
    which is not true for general distributions with $\gamma \in (0,1)$.
    Therefore, the large-$k$ expansions for the competitive ratios remain valid along any
    sequence with $k=o(n)$:
    there exist constants $c^{\tt{dp}}_\gamma,c^{\tt{ce}}_\gamma\in\R$ such that
    \[
    \acr_k(F)
    =
    1-\frac{\gamma(1-\gamma)}{2}\frac{\log k}{k}
    +\frac{c^{\tt{dp}}_\gamma}{k}
    +o\!\left(\frac{1}{k}\right),
    \qquad
    \apx_k(F)
    =
    1-\frac{\gamma(1-\gamma)}{2}\frac{\log k}{k}
    +\frac{c^{\tt{ce}}_\gamma}{k}
    +o\!\left(\frac{1}{k}\right),
    \]
    as $n\to\infty$ with $k=k(n)\to\infty$ and $k=o(n)$. Subtracting the two displays yields
    \[
    \acr_k(F)-\apx_k(F)
    =
    \frac{c^{\tt{dp}}_\gamma-c^{\tt{ce}}_\gamma}{k}
    +o\!\left(\frac{1}{k}\right).
    \]
    Since the oracle value satisfies $\mu(n,k) = \Theta\left( k^{1-\gamma}U(n)\right)$ for $\gamma\in(0,1)$, it follows that
    \[
    \vdp(n,k)-\vce(n,k)
    =
    \bigl(\acr_k(F)-\apx_k(F)\bigr)\,\mu(n,k)\,[1+o(1)]
    =
    \Theta\!\left(\frac{1}{k}\right)\cdot \Theta\!\bigl(k^{1-\gamma}U(n)\bigr)
    =
    \Omega\!\left(\Bigl(\frac{n}{k}\Bigr)^\gamma\right),
    \]
    because $U(n)=n^\gamma$ for the Pareto distribution. This proves the proposition.
\end{proof}

\section{Proofs of results in Section \ref{sec:proof}} 

\subsection{Proof of Lemma \ref{lem:dp-recursive}}
\begin{proof}
    Let \[\tau_{n,k} := \vdp(n-1,k) - \vdp(n-1,k-1).\]
    Conditioning on the first observation $X_1$, the dynamic program compares the value of accepting versus rejecting:
    \[
    \text{accept: } X_1+\vdp(n-1,k-1),
    \qquad
    \text{reject: } \vdp(n-1,k).
    \]
    Thus the optimal policy accepts $X_1$ if and only if $X_1\ge \tau_{n,k}$. Therefore,
    \begin{align}
    \vdp(n,k)
    &=\mathbb{E}\Big[\,\vdp(n-1,k)\mathbf 1\{X_1<\tau_{n,k}\}
    +\bigl(X_1+\vdp(n-1,k-1)\bigr)\mathbf 1\{X_1\ge \tau_{n,k}\}\Big]\notag\\
    &=\vdp(n-1,k)
    +\mathbb{E}\Big[\bigl(X_1-\tau_{n,k}\bigr)\mathbf 1\{X_1\ge \tau_{n,k}\}\Big]\notag\\
    &=\vdp(n-1,k)
    +\mathbb{E}\bigl[(X_1-\tau_{n,k})_+\bigr].
    \label{eq:GM-rec-aux}
    \end{align}
    To rewrite the last expectation, use the standard tail-integral identity: for any $\tau\in[0,x^\ast)$,
    \begin{align}
    \mathbb{E}[(X-\tau)_+]
    &=\int_{\tau}^{x^\ast}P(X>u)\,du
    =\int_{\tau}^{x^\ast} \bar F(u)\,du.
    \label{eq:tail-integral-identity}
    \end{align}
    Applying \eqref{eq:tail-integral-identity} to \eqref{eq:GM-rec-aux} yields
    \[
    \vdp(n,k)
    =
    \vdp(n-1,k)
    +
    \int_{\tau_{n,k}}^{x^\ast}\bar F(u)\,du,
    \]
    which is exactly the desired recursion.
\end{proof}

\subsection{Lemma~\ref{lem:dp-asymptotic-recursion} and the proof}

\begin{lemma}
\label{lem:dp-asymptotic-recursion}
Let \(\gamma\in(0,1)\), let \(U\in\mathcal{RV}_\gamma\), and set
\[
    \beta:=\frac{1-\gamma}{\gamma}.
\]
Suppose that \(R\) is nonnegative and nonincreasing, and that
\begin{equation}
\label{eq:general-R-scaling}
    \frac{nR(U(n)y)}{U(n)}
    \longrightarrow
    \frac{\gamma}{1-\gamma}y^{-\beta}
\end{equation}
locally uniformly for \(y\in(0,\infty)\).

Suppose also that \(V_0(n)\equiv0\), that
\[
    0\leq V_{j-1}(n)\leq V_j(n),
\]
and that, for every fixed \(j\geq1\),
\begin{equation}
\label{eq:general-dp-recursion}
    V_j(n)
    =
    V_j(n-1)
    +
    R\!\left(V_j(n-1)-V_{j-1}(n-1)\right).
\end{equation}
If \(V_j(n)=O(U(n))\) for every fixed \(j\), then
\[
    \frac{V_j(n)}{U(n)}
    \longrightarrow g_j,
\]
where \(g_0=0\), and \(g_j>g_{j-1}\) is the unique solution of
\[
    g_j
    =
    \frac{1}{1-\gamma}
    (g_j-g_{j-1})^{-\beta}.
\]
\end{lemma}

\begin{proof}
We argue by induction on \(j\). We prove simultaneously that
\[
    \frac{V_j(n)}{U(n)}\longrightarrow g_j
    \qquad\text{and}\qquad
    \frac{V_j(n)}{U(n)}
    -
    \frac{V_j(n-1)}{U(n-1)}
    =
    O\left(\frac1n\right).
\]
The assertion is immediate for \(j=0\).

Assume it holds for \(j-1\). By the Smooth Variation Theorem (see \cite[Theorem~1.8.2]{Bingham_Goldie_Teugels_1987}), there exists an
eventually continuously differentiable function
\(\widetilde U\sim U\) such that
\[
    \frac{n\widetilde U'(n)}{\widetilde U(n)}
    \longrightarrow\gamma.
\]
Consequently, due to Potter's bound \cite[Eq. (1.1.23)]{Borovkov_Borovkov_2008} and the dominated convergence theorem, 
\[ 
\begin{aligned} \log\frac{\widetilde U(n-1)}{\widetilde U(n)} &= -\int_{n-1}^{n} \frac{\widetilde U'(t)}{\widetilde U(t)}\,dt \\ &= -\int_{n-1}^{n} \left(\frac{\gamma}{t}+o\!\left(\frac1t\right)\right)dt = -\frac{\gamma}{n}+o\!\left(\frac1n\right). \end{aligned} 
\]
Therefore, it will further lead to 
\begin{equation}
\label{eq:smooth-U-ratio}
    \frac{\widetilde U(n-1)}{\widetilde U(n)}
    =
    1-\frac{\gamma}{n}
    +o\left(\frac1n\right).
\end{equation}
Since \(\widetilde U\sim U\), both the conclusion of the lemma and
\eqref{eq:general-R-scaling} are unchanged if \(U\) is replaced by
\(\widetilde U\).

Write
\[
    x_{i,n}:=\frac{V_i(n)}{\widetilde U(n)}.
\]
Dividing \eqref{eq:general-dp-recursion} by \(\widetilde U(n)\) gives
\begin{equation}
\label{eq:normalized-dp-recursion}
\begin{aligned}
    x_{j,n}
    &=
    \frac{\widetilde U(n-1)}{\widetilde U(n)}x_{j,n-1}  \\
    &\quad+
    \frac{
        R\!\left(
        \widetilde U(n-1)
        (x_{j,n-1}-x_{j-1,n-1})
        \right)
    }{\widetilde U(n)}.
\end{aligned}
\end{equation}
By assumption, \((x_{j,n})\) is bounded, while the induction hypothesis
gives
\[
    x_{j-1,n}\longrightarrow g_{j-1},
    \qquad
    x_{j-1,n}-x_{j-1,n-1}
    =
    O\left(\frac1n\right).
\]

We first note that the normalized marginal value remains bounded away
from zero: 
\begin{equation}
\label{eq:normalized-gap-positive}
    \liminf_{n\to\infty}
    (x_{j,n}-x_{j-1,n})>0.
\end{equation}
Indeed, let \(d_n:=x_{j,n}-x_{j-1,n}\). From \eqref{eq:smooth-U-ratio},
\eqref{eq:normalized-dp-recursion}, boundedness, and the induction
hypothesis, there exists \(C<\infty\) such that
\[
    d_n-d_{n-1}
    \geq
    -\frac{C}{n}
    +
    \frac{
        R\!\left(
        \widetilde U(n-1)d_{n-1}
        \right)
    }{\widetilde U(n)}
\]
for all sufficiently large \(n\). Choose \(\delta>0\) so small that,
\[
    \frac{\gamma}{2(1-\gamma)}(3\delta)^{-\beta}>C+1.
\]
If \(d_{n-1}\leq2\delta\), then, by the monotonicity of \(R\),
\eqref{eq:general-R-scaling}, and
\eqref{eq:smooth-U-ratio},
\[
    d_n-d_{n-1}>\frac1n
\]
for all sufficiently large \(n\). On the other hand, since \(R\geq0\),
we always have \(d_n-d_{n-1}\geq-C/n\). It follows that \(d_n\) must
eventually ($n \to \infty$) enter, and thereafter remain above, a fixed positive
neighborhood of zero. This proves \eqref{eq:normalized-gap-positive}.

The arguments of \(R\) in \eqref{eq:normalized-dp-recursion} therefore
eventually lie in a compact subset of \((0,\infty)\). Using the locally
uniform convergence in \eqref{eq:general-R-scaling},
\eqref{eq:smooth-U-ratio}, and
\(x_{j-1,n}\to g_{j-1}\), we obtain
\begin{equation}
\label{eq:normalized-drift}
    x_{j,n}-x_{j,n-1}
    =
    \frac1n
    \left[
        -\gamma x_{j,n-1}
        +
        \frac{\gamma}{1-\gamma}
        (x_{j,n-1}-g_{j-1})^{-\beta}
        +
        o(1)
    \right].
\end{equation}

Define
\[
    H_j(x)
    :=
    -\gamma x
    +
    \frac{\gamma}{1-\gamma}
    (x-g_{j-1})^{-\beta},
    \qquad x>g_{j-1}.
\]
The function \(H_j\) is continuous and strictly decreasing, with
\[
    H_j(x)\to+\infty
    \quad\text{as }x\downarrow g_{j-1},
    \qquad
    H_j(x)\to-\infty
    \quad\text{as }x\to\infty.
\]
It therefore has a unique zero with the solution \(g_j>g_{j-1}\).

Since the function $H_j$ is continuous and strictly decreasing, we can obtain, for every \(\varepsilon>0\),
the increments are uniformly positive when
\(x_{j,n-1}\leq g_j-\varepsilon\), and uniformly negative when
\(x_{j,n-1}\geq g_j+\varepsilon\), up to the common factor \(1/n\).
Since \(\sum_n n^{-1}=\infty\) and the increments in
\eqref{eq:normalized-drift} are \(O(1/n)\), the sequence cannot remain
outside \([g_j-\varepsilon,g_j+\varepsilon]\).
Hence $\varepsilon \downarrow 0$ gives,
\[
    x_{j,n}\longrightarrow g_j.
\]
Equation \eqref{eq:normalized-drift} also gives
\[
    x_{j,n}-x_{j,n-1}
    =
    O\left(\frac1n\right),
\]
which completes the induction.

Finally, \(H_j(g_j)=0\) is equivalent to
\[
    g_j
    =
    \frac{1}{1-\gamma}
    (g_j-g_{j-1})^{-\beta}.
\]
Since \(\widetilde U(n)\sim U(n)\), the same limit holds with the
original normalization \(U(n)\).
\end{proof}

\begin{remark}[Weibull counterpart]
\label{rem:weibull_counterpart}
The preceding argument has a direct counterpart when \(\gamma<0\).
Define the endpoint scale
\[
    a(n):=x^*-U(n),
\]
the value deficiency
\[
    D_j(n):=jx^*-\vdp(n,j),
\]
and the transformed tail integral
\[
    \widetilde R(s)
    :=
    \int_0^s \overline F(x^*-u)\,du.
\]
Then \(a\in\mathcal{RV}_\gamma\), the function \(\widetilde R\) is
nonnegative and nondecreasing, and the recursion for the deficiencies
is
\[
    D_j(n)
    =
    D_j(n-1)
    -
    \widetilde R\!\left(
        D_j(n-1)-D_{j-1}(n-1)
    \right).
\]

The proof proceeds as in Lemma~\ref{lem:dp-asymptotic-recursion}, after
replacing \(U\), \(V_j\), and \(R\) by \(a\), \(D_j\), and
\(\widetilde R\), respectively. The two differences are that \(a(n)\)
is decreasing and that the recursion contains a negative increment
involving a nondecreasing function. These two sign changes compensate
each other after normalization.

\end{remark}

\subsection{Proof of Proposition~\ref{prop:dp-value-gamma-le-0}}

\begin{proof}[Proof of Proposition~\ref{prop:dp-value-gamma-le-0}]
Fix \(k\geq1\). Since \(\gamma<0\), the distribution \(F\) belongs to
the Weibull domain and has a finite right endpoint \(x^*<\infty\).
Lemma~\ref{lem:dp-recursive} gives
\begin{equation}
\label{eq:dp-recursion-weibull}
\vdp(n,k)
=
\vdp(n-1,k)
+
R(\tau_{n,k}),
\qquad
R(t):=\int_t^{x^*}\overline F(u)\,du,
\end{equation}
where
\[
\tau_{n,k}
:=
\vdp(n-1,k)
-
\vdp(n-1,k-1).
\]

Define
\[
U(n):=F^{\leftarrow}\left(1-\frac1n\right),
\qquad
a(n):=x^*-U(n),
\]
and
\[
h_k(n)
:=
\frac{kx^*-\vdp(n,k)}{a(n)}.
\]
Since \(F\in\mathcal D_\gamma\) with \(\gamma<0\), the Weibull
domain-of-attraction characterization gives
\[
a(n) \in\mathcal{RV}_\gamma
\]
and, locally uniformly for \(y>0\),
\[
n\overline F\bigl(x^*-a(n)y\bigr)
\longrightarrow
y^{-1/\gamma};
\]
since $\bar F\in \mathcal{RV}_{-1/\gamma}$, see \cite[Theorem~1.2.1 and Corollary~1.2.10]
{deHaan2006extreme}.

For \(s\geq0\), set
\[
\widetilde R(s)
:=
R(x^*-s)
=
\int_0^s\overline F(x^*-u)\,du.
\]
Since \(s\mapsto\overline F(x^*-s)\) is regularly varying at zero with
index \(-1/\gamma\), Karamata's theorem at zero gives
\[
\widetilde R(s)
\sim
\frac{-\gamma}{1-\gamma}\,
s\overline F(x^*-s),
\qquad s\downarrow0.
\]
Consequently,
\begin{equation}
\label{eq:weibull-integral-scaling}
\frac{n\widetilde R(a(n)y)}{a(n)}
\longrightarrow
\frac{-\gamma}{1-\gamma}
y^{-(1-\gamma)/\gamma},
\end{equation}
locally uniformly for \(y>0\).

To rewrite the dynamic program in terms of endpoint deficiencies, let
\[
D_j(n):=jx^*-\vdp(n,j),
\qquad D_0(n)\equiv0.
\]
Since
\[
x^*-\tau_{n,j}
=
D_j(n-1)-D_{j-1}(n-1),
\]
equation \eqref{eq:dp-recursion-weibull} is equivalent to
\[
D_j(n)
=
D_j(n-1)
-
\widetilde R\!\left(
D_j(n-1)-D_{j-1}(n-1)
\right).
\]

We can now apply the Weilbull counterpart in the Remark~\ref{rem:weibull_counterpart} inductively in \(j\).
Since \(D_0(n)\equiv0\), the lemma implies that, for every fixed
\(k\geq1\),
\[
h_k(n)
=
\frac{D_k(n)}{a(n)}
\longrightarrow h_k,
\]
where \(h_0=0\) and \(h_k>h_{k-1}\) is the unique solution of
\begin{equation}
\label{eq:hk-limit-recursion-weibull}
h_k
=
\frac{1}{1-\gamma}
\bigl(h_k-h_{k-1}\bigr)^{-(1-\gamma)/\gamma}.
\end{equation}

Now define the similar $v_k$ and $z_k$ as in the proof of Proposition~\ref{prop:dp-value-gamma-ge-0}, and $z_k$ will satisfy the recursion stated in Theorem~\ref{thm:acr-dp}.

Therefore,
\[
\begin{aligned}
\vdp(n,k)
&=
kx^*-h_k a(n)[1+o(1)] \\
&=
kx^*
-
\frac{v_k}{(1-\gamma)^\gamma}
\left\{
x^*
-
F^{\leftarrow}\left(1-\frac1n\right)
\right\}
[1+o(1)].
\end{aligned}
\]
\end{proof}

\subsection{Proof of Proposition~\ref{prop:dp-value-gamma-eq-0}}

\begin{proof}
Since the CE heuristic is a feasible online policy, it cannot outperform the online optimal dynamic program policy,
\[
    \vce(n,k)
    \leq
    \vdp(n,k)
    \leq
    \mu_{n,k}.
\]
By Proposition~\ref{prop:ce-value-gamma-eq-0}, Theorem~\ref{thm:prophet}, and Remark~\ref{rem:a_n_and_u_n},
\[
    \vce(n,k)
    =
    kU(n)[1+o(1)],
    \qquad
    \mu_{n,k}
    =
    kU(n)[1+o(1)].
\]
The result follows by the squeeze theorem.
\end{proof}

\subsection{Proof of Lemma~\ref{lem:ce-recursive}}

\begin{proof}
The standard myopic fluid problem with \(t\) periods remaining and
\(j\) selection opportunities remaining is
\[
\max_{x(\cdot)} \quad \mathbb E[Xx(X)]
\qquad
\text{s.t.}\quad
\mathbb E[x(X)]\le \frac jt,\qquad
x(X)\in\{0,1\}.
\]
Here \(X\sim F\), and \(x(X)\) is a measurable acceptance rule.

\noindent For a multiplier \(\mu\ge 0\), the Lagrangian is
\[
\mathcal L(x,\mu)
=
\mathbb E[Xx(X)]
+\mu\left(\frac jt-\mathbb E[x(X)]\right)
=
\mu\frac jt+\mathbb E[(X-\mu)x(X)].
\]
Maximizing pointwise over \(x(X)\in\{0,1\}\) gives
\[
\phi(\mu)
=
\mu\frac jt+\mathbb E[(X-\mu)^+].
\]
Thus the dual problem is \(\inf_{\mu\ge 0}\phi(\mu)\).

\noindent If \(F\) is continuous, then \(\phi\) is differentiable and
\[
\phi'(\mu)=\frac jt-\mathbb P(X>\mu).
\]
Hence the optimal dual price satisfies
\[
\mathbb P(X>\mu^*)=\frac jt,
\]
or equivalently
\[
\mu^*=F^{\leftarrow}\left(1-\frac jt\right).
\]
Therefore the certainty-equivalent policy accepts a candidate in state
\((t,j)\) iff
\[
X\ge q_{t,j}:=F^{\leftarrow}\left(1-\frac jt\right).
\]
In particular, at the initial state \((n,k)\),
\[
q_{n,k}=F^{\leftarrow}\left(1-\frac kn\right).
\]
Let $q:=q_{n,k}$. So under the CE policy, we accept the first observation $X_1$ iff $X_1>q$.
Conditioning on this event,
\[
\vce(n,k)
=P(X_1>q)\Big(\mathbb{E}[X_1\mid X_1>q]+\vce(n-1,k-1)\Big)
+P(X_1\le q)\,\vce(n-1,k).
\]
Since $q=F^{\leftarrow}(1-k/n)$, we have $P(X_1>q)=\bar F(q)=k/n$. Moreover,
\[
\mathbb{E}[X_1\mathbf 1\{X_1>q\}]
=\int_q^{x^\ast} u f(u)\,du
= q\,\bar F(q)+\int_q^{x^\ast} \bar F(u)\,du,
\]
where the last identity follows by integration by parts. Therefore
\[
P(X_1>q)\,\mathbb{E}[X_1\mid X_1>q]
=\mathbb{E}[X_1\mathbf 1\{X_1>q\}]
=\frac{k}{n}\,q+\int_q^{x^\ast}\bar F(u)\,du.
\]
Substituting these expressions into the conditional decomposition yields the stated recursion.
\end{proof}

\subsection{Lemma~\ref{lem:linear-asymptotic-recursion} and the proof.}

\begin{lemma}[Linear asymptotic recursion]
\label{lem:linear-asymptotic-recursion}
Let \(a>0\), and suppose that a bounded sequence \((x_n)\) satisfies
\[
    x_n-x_{n-1}
    =
    \frac{1}{n}\bigl(-a x_{n-1}+b+o(1)\bigr).
\]
Then
\[
    x_n\longrightarrow \frac{b}{a}.
\]
\end{lemma}

\begin{proof}
Let \(x^\ast=b/a\). For every \(\varepsilon>0\), the increment is
strictly negative, uniformly up to the factor \(1/n\), whenever
\(x_{n-1}\geq x^\ast+\varepsilon\), and strictly positive whenever
\(x_{n-1}\leq x^\ast-\varepsilon\). Since
\(\sum_n n^{-1}=\infty\) and the increments are \(O(1/n)\), the sequence
must eventually enter the neighborhood \([x^\ast-\varepsilon,x^\ast + \varepsilon]\), and letting $\varepsilon\downarrow 0$ gives \(x_n\to x^\ast\).
\end{proof}

\subsection{Proof of Proposition~\ref{prop:ce-value-gamma-le-0}}

\begin{proof}
We first establish the leading-order conclusion for every
\(\gamma<0\). Under the CE policy, when \(j\) units remain over \(t\) periods, the
conditional probability of acceptance is \(j/t\). This is exactly the sampling-without-replacement rule for generating the uniformly generated $j$ subset of the remaining $t$ periods. Consequently, starting from the state $(k,n)$, the
set of the \(k\) acceptance times has the same distribution as a
uniformly chosen \(k\)-subset of the \(n\) arrival times.

Fix \(\varepsilon\in(0,1)\), and let \(m_n=\lfloor\varepsilon n\rfloor\).
Let \(A_{n,\varepsilon}\) be the event that all \(k\) acceptances occur
before the final \(m_n\) arrivals. Then
\begin{align*}
\mathbb P(A_{n,\varepsilon})
=
\frac{\binom{n-m_n}{k}}{\binom{n}{k}}
= \prod_{r=0}^{k-1} \frac{n-m_n-r}{n-r} = \prod_{r=0}^{k-1} \frac{1-\frac{m_n}{n}-\frac{r}{n}}{1-\frac{r}{n}} = \longrightarrow
(1-\varepsilon)^k.
\end{align*}
On \(A_{n,\varepsilon}\), every acceptance occurs at a state with
\(t\geq m_n+1\) periods and \(1\leq j\leq k\) units remaining.
Therefore, the corresponding CE threshold satisfies
\[
    F^{\leftarrow}\left(1-\frac jt\right)
    =
    U\left(\frac tj\right)
    \geq
    U\left(\frac{m_n+1}{k}\right).
\]
Since rewards are nonnegative, it follows that
\[
    \vce(n,k)
    \geq
    kU\left(\frac{m_n+1}{k}\right)
    \mathbb P(A_{n,\varepsilon}).
\]
Since \(U(t)\to x^*\) as \(t\to\infty\),
\[
    \liminf_{n\to\infty}\vce(n,k)
    \geq
    kx^*(1-\varepsilon)^k.
\]
Letting \(\varepsilon\downarrow0\), and using the trivial upper bound
\(\vce(n,k)\leq kx^*\), gives
\[
    \vce(n,k)\longrightarrow kx^*
\]
for every \(\gamma<0\).

We now derive the sharper endpoint-scale expansion. For the remainder
of the proof, suppose in addition that \(-1<\gamma<0\).

Set
\[
    a(n):=x^*-U(n),
    \qquad
    q_{n,k}:=F^{\leftarrow}\!\left(1-\frac{k}{n}\right).
\]
Lemma~\ref{lem:ce-recursive} gives
\begin{equation}
\label{eq:ce-rec-neg}
\begin{aligned}
\vce(n,k)
&=
\frac{k}{n}\vce(n-1,k-1)
+\int_{q_{n,k}}^{x^*}\bar F(u)\,du \\
&\quad
+\frac{k}{n}q_{n,k}
+\left(1-\frac{k}{n}\right)\vce(n-1,k).
\end{aligned}
\end{equation}

Define the deficiency
\[
    D(n,k):=kx^*-\vce(n,k).
\]
Subtracting \eqref{eq:ce-rec-neg} from \(kx^*\) gives
\begin{equation}
\label{eq:D-rec-neg}
\begin{aligned}
D(n,k)
&=
\left(1-\frac{k}{n}\right)D(n-1,k)
+\frac{k}{n}D(n-1,k-1) \\
&\quad
+\frac{k}{n}(x^*-q_{n,k})
-\int_{q_{n,k}}^{x^*}\bar F(u)\,du .
\end{aligned}
\end{equation}

Since \(a\in\mathcal{RV}_\gamma\), the Smooth Variation Theorem
\cite[Theorem~1.8.2]{Bingham_Goldie_Teugels_1987} yields an eventually
continuously differentiable function \(\widetilde a\sim a\) such that
\[
    \frac{n\widetilde a'(n)}{\widetilde a(n)}
    \longrightarrow\gamma.
\]
Consequently,
\begin{equation}
\label{eq:an-ratio-ce-neg}
    \frac{\widetilde a(n-1)}{\widetilde a(n)}
    =
    1-\frac{\gamma}{n}
    +o\left(\frac1n\right).
\end{equation}

Define
\[
    \widehat h_j(n):=\frac{D(n,j)}{\widetilde a(n)},
    \qquad j\geq0,
\]
with \(\widehat h_0(n)\equiv0\). Regular variation and
\(\widetilde a(n)\sim a(n)\) give
\begin{equation}
\label{eq:q-ratio-ce-neg}
    \frac{x^*-q_{n,k}}{\widetilde a(n)}
    =
    \frac{a(n/k)}{\widetilde a(n)}
    =
    k^{-\gamma}[1+o(1)].
\end{equation}

By Karamata's theorem at zero,
\[
    \int_q^{x^*}\bar F(u)\,du
    =
    \frac{-\gamma}{1-\gamma}
    (x^*-q)\bar F(q)[1+o(1)],
    \qquad q\uparrow x^*.
\]
Since \(\bar F(q_{n,k})=k/n\), it follows that
\begin{equation}
\label{eq:tailint-ce-neg}
    \int_{q_{n,k}}^{x^*}\bar F(u)\,du
    =
    \frac{k}{n}\frac{-\gamma}{1-\gamma}
    (x^*-q_{n,k})[1+o(1)].
\end{equation}
Therefore,
\[
\frac{k}{n}(x^*-q_{n,k})
-\int_{q_{n,k}}^{x^*}\bar F(u)\,du
=
\frac{k}{n}\frac{1}{1-\gamma}
(x^*-q_{n,k})[1+o(1)].
\]

Dividing \eqref{eq:D-rec-neg} by \(\widetilde a(n)\), and using
\eqref{eq:an-ratio-ce-neg}--\eqref{eq:tailint-ce-neg}, gives
\begin{equation}
\label{eq:gk-increment-ce-neg}
\begin{aligned}
\widehat h_k(n)-\widehat h_k(n-1)
=
\frac1n\left[
    -(k+\gamma)\widehat h_k(n-1)
    +k\widehat h_{k-1}(n-1)
    +\frac{k^{1-\gamma}}{1-\gamma}
    +o(1)
\right].
\end{aligned}
\end{equation}

We now argue inductively in \(k\) to show the convergence of $\hat h_k(n)$. The assertion is immediate for
\(k=0\). Suppose that
\[
    \widehat h_{k-1}(n)\longrightarrow w_{k-1}.
\]
Because \(\gamma>-1\), we have \(k+\gamma>0\) for every \(k\geq1\).
Moreover, \eqref{eq:gk-increment-ce-neg} and the induction hypothesis
imply that \((\widehat h_k(n))\) is bounded. Hence
Lemma~\ref{lem:linear-asymptotic-recursion} yields
\[
    \widehat h_k(n)\longrightarrow w_k,
\]
where
\[
    w_k
    =
    \frac{k}{k+\gamma}w_{k-1}
    +
    \frac{k^{1-\gamma}}
    {(k+\gamma)(1-\gamma)}.
\]
With \(w_0=0\), this gives
\[
    w_1=\frac{1}{(1+\gamma)(1-\gamma)}
    =\frac{1}{1-\gamma^2}.
\]

Finally, since \(\widetilde a(n)\sim a(n)\),
\[
\begin{aligned}
\vce(n,k)
&=
kx^*-D(n,k) \\
&=
kx^*
-w_k\left\{
x^*-F^{\leftarrow}\!\left(1-\frac1n\right)
\right\}[1+o(1)].
\end{aligned}
\]
\end{proof}

\subsection{Proof of Proposition~\ref{prop:ce-value-gamma-eq-0}}

\begin{proof}

From the same reasoning as the proof of Proposition~\ref{prop:ce-value-gamma-le-0}, we obtain
\[
    \vce(n,k)
    \geq
    kU\left(\frac{m_n+1}{k}\right)
    \mathbb P(A_{n,\varepsilon})
    \geq
    kU\left(\frac{\varepsilon n}{k}\right)
    \mathbb P(A_{n,\varepsilon})
\]
For \(F\in\mathcal D_0\), the Gumbel quantile relation gives, for every
fixed \(c>0\),
\[
    U(cn)=U(n)+a(n)\log c+o(a(n)),
\]
where \(a(n)=o(U(n))\). Hence
\[
    \frac{U(cn)}{U(n)}\longrightarrow1.
\]
Dividing the preceding lower bound by \(U(n)\) therefore yields
\[
    \liminf_{n\to\infty}
    \frac{\vce(n,k)}{U(n)}
    \geq
    k(1-\varepsilon)^k.
\]
Letting \(\varepsilon\downarrow0\), we obtain
\[
    \liminf_{n\to\infty}
    \frac{\vce(n,k)}{U(n)}
    \geq k.
\]

On the other hand, the CE policy cannot outperform the prophet, so
\[
    \vce(n,k)\leq\mu_{n,k}.
\]
By Theorem~\ref{thm:prophet},
\[
    \mu_{n,k}=kU(n)+O(a(n))=kU(n)[1+o(1)].
\]
It follows that
\[
    \limsup_{n\to\infty}
    \frac{\vce(n,k)}{U(n)}
    \leq k.
\]
Combining the upper and lower bounds proves the result.
\end{proof}

\subsection{Proof of Theorem~\ref{thm:prophet}}

\begin{proof}

Recall that
\[
\mu_{n,k}
=
\E\left[\sum_{r=1}^k X_{n-r+1:n}\right],
\]
and that we can represent \(X_{n-r+1:n}\) as
\(F^{\leftarrow}(V_{n-r+1:n})\), where \(V_{n-r+1:n}\) is the
\(r\)-th largest order statistic from a sample of independent uniform
random variables.

From the standard theory of order statistics,
\[
1-V_{n-r+1:n}\sim \operatorname{Beta}(r,n-r+1).
\]
Define
\[
S_{n,r}:=1-V_{n-r+1:n},
\]
which has the same distribution as the \(r\)-th smallest uniform order
statistic. Also, let
\[
W_{n,r}:=\frac{1}{nS_{n,r}}.
\]
It is then easy to verify that
\[
X_{n-r+1:n}
=
F^{\leftarrow}(V_{n-r+1:n})
=
U\left(\frac{1}{S_{n,r}}\right)
=
U(nW_{n,r}).
\]

Furthermore, by the exponential representation of uniform spacings
\[
S_r:=V_{r:n}-V_{r-1:n},
\qquad r=1,\dots,n+1,
\]
where \(V_{0:n}=0\) and \(V_{n+1:n}=1\), we have
(see \cite[Theorem~2.2, p.~208]{devroye_1986})
\[
\left(S_1,\dots,S_{n+1}\right)
\overset{d}{=}
\left(
\frac{E_1}{\sum_{i=1}^{n+1}E_i},
\dots,
\frac{E_{n+1}}{\sum_{i=1}^{n+1}E_i}
\right).
\]
This directly implies that
\[
S_{n,r}
\overset{d}{=}
V_{r:n}
\overset{d}{=}
\frac{\sum_{i=1}^r E_i}{\sum_{i=1}^{n+1}E_i}
=
\frac{\Gamma_r}{\Gamma_{n+1}},
\]
where
\[
\Gamma_r:=E_1+\cdots+E_r,
\qquad
E_i\sim\operatorname{Exp}(1).
\]

Therefore,
\[
nS_{n,r}
\overset{d}{=}
\frac{n}{\Gamma_{n+1}}\Gamma_r
\implies
\Gamma_r,
\]
because
\[
\frac{n}{\Gamma_{n+1}}\to 1
\]
by the law of large numbers. Hence, by the continuous mapping theorem,
\[
W_{n,r}\implies\Gamma_r^{-1}.
\]

We next establish moment bounds for \(W_{n,r}\). For \(0<p<r\),
\[
\E[W_{n,r}^{p}]
=
n^{-p}\E[S_{n,r}^{-p}],
\]
where
\[
S_{n,r}\sim\operatorname{Beta}(r,n-r+1).
\]
From the integral representation of the beta distribution, and we use $\text{B}(\cdot)$ for shorthand,
\[
\E[S_{n,r}^{-p}]
=
\frac{\operatorname{B}(r-p,n-r+1)}
     {\operatorname{B}(r,n-r+1)}.
\]
Using
\[
\operatorname{B}(a,b)
=
\frac{\Gamma(a)\Gamma(b)}{\Gamma(a+b)},
\]
we obtain
\[
\E[S_{n,r}^{-p}]
=
\frac{\Gamma(r-p)\Gamma(n+1)}
     {\Gamma(r)\Gamma(n+1-p)}.
\]
Consequently,
\[
\E[W_{n,r}^{p}]
=
n^{-p}
\frac{\Gamma(r-p)\Gamma(n+1)}
     {\Gamma(r)\Gamma(n+1-p)}.
\]
Using the gamma-ratio asymptotic,
\[
\frac{\Gamma(n+1)}{\Gamma(n+1-p)}
\sim n^p,
\]
we obtain
\[
\E[W_{n,r}^{p}]
\to
\frac{\Gamma(r-p)}{\Gamma(r)}.
\]
In particular,
\[
\sup_n \E[W_{n,r}^{p}]<\infty,
\qquad 0<p<r.
\]

\subsubsection*{Case 1: \(\gamma\in(0,1)\)}

Since \(U\in\mathcal{RV}_{\gamma}\), for every fixed \(x>0\),
\[
\frac{U(tx)}{U(t)}
\to x^\gamma
\qquad\text{as }t\to\infty.
\]
Since
\[
X_{n-r+1:n}=U(nW_{n,r})
\]
and
\[
W_{n,r}\implies\Gamma_r^{-1},
\]
it follows that
\[
\frac{X_{n-r+1:n}}{U(n)}
=
\frac{U(nW_{n,r})}{U(n)}
\implies
\Gamma_r^{-\gamma}.
\]

We now establish uniform integrability in order to obtain convergence
of expectations. Choose \(\epsilon>0\) such that
\[
\gamma+\epsilon<1.
\]
By Potter's bound
(see \cite{Bingham_Goldie_Teugels_1987}), for all sufficiently large
\(n\),
\[
\frac{U(nW_{n,r})}{U(n)}
\leq
C
\begin{cases}
W_{n,r}^{\gamma-\epsilon}, & W_{n,r}\leq 1,\\[2mm]
W_{n,r}^{\gamma+\epsilon}, & W_{n,r}>1.
\end{cases}
\]
Choose \(q>1\) sufficiently close to \(1\) such that
\[
q(\gamma+\epsilon)<1\leq r.
\]
Using the preceding moment bound, we obtain
\[
\sup_n
\E\left[
\left(
\frac{X_{n-r+1:n}}{U(n)}
\right)^q
\right]
=
\sup_n
\E\left[
\left(
\frac{U(nW_{n,r})}{U(n)}
\right)^q
\right]
<\infty.
\]

Now define
\[
Y_n:=\frac{X_{n-r+1:n}}{U(n)}.
\]
We claim that \(\{Y_n\}\) is uniformly integrable. Indeed,
\begin{align*}
\lim_{M\to\infty}
\sup_n
\E\left[
Y_n\mathbb{I}\{Y_n>M\}
\right]
&\leq
\lim_{M\to\infty}
\sup_n
\E\left[
\frac{Y_n^q}{M^{q-1}}
\mathbb{I}\{Y_n>M\}
\right] \\
&\leq
\lim_{M\to\infty}
\frac{1}{M^{q-1}}
\sup_n \E[Y_n^q] \\
&=0,
\end{align*}
where the last equality follows from
\[
\sup_n\E[Y_n^q]<\infty.
\]

By uniform integrability, we obtain
\[
\frac{\E[X_{n-r+1:n}]}{U(n)}
\to
\E[\Gamma_r^{-\gamma}].
\]
Since \(\Gamma_r\) has a \(\operatorname{Gamma}(r,1)\) distribution,
its density is
\[
f_{\Gamma_r}(x)
=
\frac{x^{r-1}e^{-x}}{\Gamma(r)},
\qquad x>0.
\]
Therefore,
\[
\E[\Gamma_r^{-\gamma}]
=
\frac{1}{\Gamma(r)}
\int_0^\infty x^{r-\gamma-1}e^{-x}\,dx
=
\frac{\Gamma(r-\gamma)}{\Gamma(r)}.
\]
Hence,
\[
\E[X_{n-r+1:n}]
=
U(n)\frac{\Gamma(r-\gamma)}{\Gamma(r)}
+
o(U(n)).
\]
Summing over \(r=1,\dots,k\), we obtain
\[
\mu_{n,k}
=
U(n)
\sum_{r=1}^k
\frac{\Gamma(r-\gamma)}{\Gamma(r)}
+
o(U(n)).
\]

\subsubsection*{Case 2: $\gamma <0$}
Let $x^\ast$ denote the finite right endpoint of \(F\), and define the function
\[
d(t):=x^\ast-U(t).
\]
Since \(F\in D_\gamma\) with \(\gamma<0\), we have
\[
d\in\operatorname{RV}_\gamma,
\]
which means that, for every fixed \(x>0\),
\[
\frac{d(tx)}{d(t)}
\to x^\gamma
\qquad\text{as }t\to\infty.
\]
Then similarly to Case 1, we obtain 
\[
\frac{d(n W_{n,r})}{d(n)} \implies \Gamma_r^{-\gamma},
\]
and also, we have 
\[
\sup_n\E[W_{n,r}^{-p}]<\infty
\qquad\text{for every }p>0.
\]
Choose \(\epsilon>0\) such that
\[
\gamma+\epsilon<0.
\]
By Potter's bound, for all sufficiently large \(n\),
\[
\frac{d(nW_{n,r})}{d(n)}
\leq
C
\begin{cases}
W_{n,r}^{\gamma-\epsilon},
    & W_{n,r}\leq 1,\\[2mm]
W_{n,r}^{\gamma+\epsilon},
    & W_{n,r}>1.
\end{cases}
\]
When \(W_{n,r}>1\), since \(\gamma+\epsilon<0\),
\[
W_{n,r}^{\gamma+\epsilon}\leq 1.
\]
When \(W_{n,r}\leq 1\),
\[
W_{n,r}^{\gamma-\epsilon}
=
W_{n,r}^{-(|\gamma|+\epsilon)}.
\]
Choose \(q>1\). Using the preceding negative-moment bound, we obtain
\[
\sup_n
\E\left[
\left(
\frac{x^\ast-X_{n-r+1:n}}{d(n)}
\right)^q
\right]
<\infty.
\]
Therefore, similar to case 1, the sequence
\[
\left\{
\frac{x^\ast-X_{n-r+1:n}}{d(n)}
\right\}_{n\geq 1}
\]
is uniformly integrable. Hence,
\[
\frac{\E[x^\ast-X_{n-r+1:n}]}{d(n)}
\to
\E[\Gamma_r^{-\gamma}].
\]
Therefore, we obtain
\[
\mu_{n,k}
=
kx^\ast
-
d(n)
\sum_{r=1}^k
\frac{\Gamma(r-\gamma)}{\Gamma(r)}
+
o(d(n)).
\]

\subsubsection*{Case 3: $\gamma =0$}

Suppose that \(F\in D_0\). Then there exists an auxiliary function
\(a(t)>0\) such that, for every fixed \(x>0\),
\[
\frac{U(tx)-U(t)}{a(t)}
\to
\log x
\qquad\text{as }t\to\infty.
\]
A convenient asymptotically equivalent choice is
\[
a_n:=U(en)-U(n).
\]
Indeed, by setting \(x=e\), we obtain
\[
\frac{U(en)-U(n)}{a(n)}
\to 1,
\]
and hence
\[
a_n\sim a(n).
\]

Since
\[
X_{n-r+1:n}=U(nW_{n,r})
\]
and
\[
W_{n,r}\implies\Gamma_r^{-1},
\]
we obtain
\[
\frac{X_{n-r+1:n}-U(n)}{a(n)}
=
\frac{U(nW_{n,r})-U(n)}{a(n)}
\implies
\log(\Gamma_r^{-1})
=
-\log\Gamma_r.
\]

We next establish uniform integrability. The extended Potter bound (see \cite{deHaan2006extreme}, Theorem B.2.18, p. 383) for
Gumbel-domain quantile functions implies that, for every sufficiently
small \(\epsilon>0\) and all sufficiently large \(n\),
\[
\left|
\frac{U(nW_{n,r})-U(n)}{a(n)}
\right|
\leq
C\left(
W_{n,r}^{\epsilon}
+
W_{n,r}^{-\epsilon}
\right).
\]
Choose \(q>1\) and \(\epsilon>0\) such that
\[
q\epsilon<1\leq r.
\]
Using the positive and negative moment bound from case 1 and 2
\[
\sup_n\E[W_{n,r}^{q\epsilon}]<\infty, \quad \sup_n\E[W_{n,r}^{-q\epsilon}]<\infty,
\]
we obtain
\[
\sup_n
\E\left[
\left|
\frac{X_{n-r+1:n}-U(n)}{a(n)}
\right|^q
\right]
<\infty.
\]
Therefore, followed by the similar reasoning in the previous cases, the sequence
\[
\left\{
\frac{X_{n-r+1:n}-U(n)}{a(n)}
\right\}_{n\geq 1}
\]
is uniformly integrable. Consequently,
\[
\frac{\E[X_{n-r+1:n}]-U(n)}{a(n)}
\to
-\E[\log\Gamma_r].
\]

Since \(\Gamma_r\sim\operatorname{Gamma}(r,1)\),
\begin{align*}
\E[\log\Gamma_r]
&=
\frac{1}{\Gamma(r)}
\int_0^\infty
(\log x)x^{r-1}e^{-x}\,dx\\
&=
\frac{\Gamma'(r)}{\Gamma(r)}
=
\psi(r),
\end{align*}
where
\[
\psi(r):=\frac{\Gamma'(r)}{\Gamma(r)}
\]
is the digamma function. Hence,
\[
\E[X_{n-r+1:n}]
=
U(n)-a(n)\psi(r)+o(a(n)).
\]

Summing over \(r=1,\dots,k\), we obtain
\[
\mu_{n,k}
=
kU(n)
+
a(n)
\sum_{r=1}^k[-\psi(r)]
+
o(a(n)).
\]
Therefore, we obtain
\[
\mu_{n,k}
=
kF^{\leftarrow}\left(1-\frac{1}{n}\right)
+
a_n
\sum_{r=1}^k[-\psi(r)]
+
o(a_n),
\]
where
\[
a_n
=
U(en)-U(n).
\]

\begin{remark}
\label{rem:a_n_and_u_n}
We further claim $a_n = o(U(n))$. Indeed, since for every fixed $c>0$,
\[
\frac{U(cx) - U(x)}{a_n} \longrightarrow \log c.
\]
Take $c = e^{-M}$, where $M>0$ is arbitrary, we obtain
\[
\frac{U(n) - U(ne^{-M})}{a_n} \longrightarrow \log c,
\]
which leads to 
\[
\liminf_{n\to \infty} \frac{U(n)}{a_n} \ge M,
\]
but $M$ can be arbitrarily large, so 
\[
\frac{a_n}{U(n)} \longrightarrow 0.
\]
\end{remark}

\begin{lemma}
\label{lem:gamma-sum-identity}
Let $k\in\mathbb{N}$, $k\ge 1$, and let $\gamma\in\mathbb{R}$ be such that
$\Gamma(r-\gamma)$ is finite for all $r=1,\dots,k$ (e.g.\ it suffices that $\gamma<1$).
Then
\[
\sum_{r=1}^{k}\frac{\Gamma(r-\gamma)}{\Gamma(r)}
\;=\;
\frac{\Gamma(k+1-\gamma)}{(1-\gamma)\Gamma(k)}.
\]
\end{lemma}

\begin{proof}
For $r\ge 2$, define
\[
A_r := \frac{\Gamma(r+1-\gamma)}{\Gamma(r)}.
\]
Using the Gamma recursion $\Gamma(r+1-\gamma)=(r-\gamma)\Gamma(r-\gamma)$ and
$\Gamma(r)=(r-1)\Gamma(r-1)$, we compute
\begin{align*}
A_r - A_{r-1}
&=
\frac{\Gamma(r+1-\gamma)}{\Gamma(r)}-\frac{\Gamma(r-\gamma)}{\Gamma(r-1)} \\
&=
\frac{(r-\gamma)\Gamma(r-\gamma)}{\Gamma(r)}-\frac{(r-1)\Gamma(r-\gamma)}{\Gamma(r)} \\
&=
(1-\gamma)\,\frac{\Gamma(r-\gamma)}{\Gamma(r)}.
\end{align*}
Hence, for $k\ge 2$,
\[
(1-\gamma)\sum_{r=2}^{k}\frac{\Gamma(r-\gamma)}{\Gamma(r)}
=\sum_{r=2}^{k}(A_r-A_{r-1})=A_k-A_1.
\]
Now note that
\[
A_1=\frac{\Gamma(2-\gamma)}{\Gamma(1)}
=(1-\gamma) \Gamma(1-\gamma).
\]
Therefore,
\[
(1-\gamma)\sum_{r=1}^{k}\frac{\Gamma(r-\gamma)}{\Gamma(r)}
=
(1-\gamma)\frac{\Gamma(1-\gamma)}{\Gamma(1)} + (A_k-A_1)
= A_k.
\]
So it gives
\[
\sum_{r=1}^{k}\frac{\Gamma(r-\gamma)}{\Gamma(r)}
=
\frac{1}{1-\gamma}\cdot \frac{\Gamma(k+1-\gamma)}{\Gamma(k)},
\]
as claimed.
\end{proof}
Thus, using this lemma~\ref{lem:gamma-sum-identity}, we will get the results as stated in Theorem~\ref{thm:prophet}.
\end{proof}

\end{document}